\pdfoutput=1

\documentclass[hidelinks,11pt]{amsart}
\usepackage[letterpaper,margin=1in]{geometry}

\usepackage{lmodern}
\usepackage{microtype}

\usepackage[utf8]{inputenc}

\usepackage[inline,shortlabels]{enumitem}

\usepackage{amssymb}
\usepackage{mathtools,stmaryrd,mathrsfs}
\usepackage{thmtools,thm-restate}
\usepackage{colonequals}
\usepackage[retainorgcmds]{IEEEtrantools}
\usepackage{tikz-cd}

\usepackage[pdfusetitle]{hyperref}
\usepackage[hyphenbreaks]{breakurl}

\usepackage{tabularx}
\newcommand{\N}{\mathbf{N}}
\newcommand{\Z}{\mathbf{Z}}
\newcommand{\Q}{\mathbf{Q}}
\newcommand{\R}{\mathbf{R}}
\newcommand{\C}{\mathbf{C}}
\newcommand{\F}{\mathbf{F}}

\newcommand{\PP}{\mathbf{P}}

\newcommand{\cA}{\mathscr{A}}
\newcommand{\cC}{\mathscr{C}}
\newcommand{\cD}{\mathscr{D}}

\newcommand{\cH}{\mathscr{H}}
\newcommand{\cJ}{\mathscr{J}}

\newcommand{\cM}{\mathscr{M}}
\newcommand{\cO}{\mathscr{O}}

\newcommand{\cT}{\mathscr{T}}
\newcommand{\cU}{\mathscr{U}}

\newcommand{\cW}{\mathscr{W}}
\newcommand{\cX}{\mathscr{X}}
\newcommand{\cY}{\mathscr{Y}}

\newcommand{\fA}{\mathfrak{A}}
\newcommand{\fC}{\mathfrak{C}}

\newcommand{\fm}{\mathfrak{m}}

\newcommand{\cc}{\mathrm{c}}
\newcommand{\pD}{{^p\mathrm{D}}}
\newcommand{\perf}{\mathrm{perf}}
\newcommand{\perfd}{\mathrm{perfd}}
\newcommand{\Rr}{\mathrm{R}}

\newcommand{\suchthat}{\;\ifnum\currentgrouptype=16 \middle\fi\vert\;}
\newcommand{\abs}[1]{\lvert#1\rvert}
\newcommand\restr[2]{{\left.\kern-\nulldelimiterspace#1\vphantom{\big|}\right|_{#2}}}

\DeclarePairedDelimiter\floor{\lfloor}{\rfloor}

\DeclareMathOperator{\ad}{ad}
\DeclareMathOperator{\Arr}{Arr}
\DeclareMathOperator{\Aut}{Aut}
\DeclareMathOperator{\Bl}{Bl}
\DeclareMathOperator{\Cc}{C}
\DeclareMathOperator{\charac}{char}

\DeclareMathOperator*{\colim}{colim}

\DeclareMathOperator{\D}{D}

\DeclareMathOperator{\et}{\acute{e}t}
\DeclareMathOperator{\Ext}{Ext}
\DeclareMathOperator{\Fun}{Fun}
\DeclareMathOperator{\GL}{GL}
\DeclareMathOperator{\Hh}{H}
\DeclareMathOperator*{\hocolim}{hocolim}

\DeclareMathOperator{\id}{id}
\DeclareMathOperator{\im}{im}

\DeclareMathOperator{\Lan}{Lan}
\DeclareMathOperator{\liset}{lis-\acute{e}t}
\DeclareMathOperator{\Mor}{Mor}
\DeclareMathOperator{\op}{op}

\DeclareMathOperator{\qproet}{qpro\acute{e}t}
\DeclareMathOperator{\rad}{rad}
\DeclareMathOperator{\red}{red}
\DeclareMathOperator{\sh}{sh}
\DeclareMathOperator{\Spa}{Spa}
\DeclareMathOperator{\Spec}{Spec}

\DeclareMathOperator{\Spd}{Spd}
\DeclareMathOperator{\supp}{supp}
\DeclareMathOperator{\trdeg}{trdeg}

\theoremstyle{plain}
\newtheorem{thm}{Theorem}[section]

\newtheorem{prop}[thm]{Proposition}
\newtheorem{lem}[thm]{Lemma}
\newtheorem{cor}[thm]{Corollary}

\theoremstyle{definition}
\newtheorem{defn}[thm]{Definition}
\newtheorem{exmp}[thm]{Example}
\newtheorem{sit}[thm]{Situation}

\theoremstyle{remark}
\newtheorem{rem}[thm]{Remark}

\begin{document}

\title{The cohomology of the moduli space of curves at infinite level}
\author{Emanuel Reinecke}
\thanks{This material is based upon work supported by the National Science Foundation under grant no.\ DMS-1501461, DMS-1801689 and DMS-1440140, and by a Rackham Predoctoral Fellowship.
Parts of this paper were written while the author was in residence at the Mathematical Sciences Research Institute in Berkeley, California, during the spring of 2019.}

\address{Department of Mathematics\\University of Michigan\\Ann Arbor, MI 48109, USA}
\email{\href{mailto:reinec@umich.edu}{reinec@umich.edu}}
\urladdr{\url{http://www-personal.umich.edu/~reinec/}}

\begin{abstract}
  Full level-$n$ structures on smooth, complex curves are trivializations of the $n$-torsion points of their Jacobians.
  We give an algebraic proof that the \'etale cohomology of the moduli space of smooth, complex curves of genus at least $2$ with ``infinite level structure'' vanishes in degrees above $4g-5$.
  This yields a new perspective on a result of Harer who showed such vanishing already at finite level via topological methods.
  We obtain similar results for moduli spaces of stable curves and curves of compact type which are not covered by Harer's methods.
  The key ingredients in the proof are a vanishing statement for certain constructible sheaves on perfectoid spaces and a comparison of the \'etale cohomology of different towers of Deligne--Mumford stacks in the presence of ramification.
\end{abstract}

\maketitle

\section{Introduction}\label{sect:intro}

Let $\cM_g$ be the moduli space of smooth, complex curves of genus $g \ge 2$.
A strong interest in the rational singular cohomology of $\cM_g$ dates back as far as the 1980s.
Its computation could, for example, help to understand the enumerative geometry of $\cM_g$ (cf.\ e.g.\ \cite{MR717614}) or predict the global structure of $\cM_g$ (e.g., Looijenga's conjecture that the coarse space $M_g$ can be covered by $g-1$ open affine subschemes).
Unfortunately, only little is known for general $g$.
In low degrees, genus-independent computations of $\Hh^i(\cM_g(\C),\Q)$ are due to Mumford \cite[Thm.~1]{MR0219543} ($i=1$) and Harer \cite{MR700769} ($i=2$) \cite{MR1106936} ($i=2,3$).
At the other end of the range, Harer showed the following.
\begin{thm}[{\cite[Cor.~4.3]{MR830043}}]\label{thm:Harer}
  For $i > 4g-5$, we have $\Hh^i(\cM_g(\C),\Q) = 0$.
\end{thm}
A common thread of all these computations is the pervasive use of tools from geometric topology and Teichm\"uller theory, using the construction of $\cM_g$ as the quotient of contractible Teichm\"uller space $\cT_g$ by the mapping class group $\Gamma_g$.
There has been some interest in finding proofs that rely on more algebro-geometric methods; see e.g.\ \cite[p.~251]{MR1648075}.
For example, Arbarello--Cornalba realized that the results on cohomology in low degree follow via Hodge theory from Theorem~\ref{thm:Harer} \cite{MR1733327} \cite[Thm.~10]{MR2655321}.
The goal of this paper is to give a new, essentially algebro-geometric perspective on Theorem~\ref{thm:Harer}.

\subsection*{Our results}
In \cite{MR830043}, Harer showed in fact that the virtual cohomological dimension of $\Gamma_g$ is $4g-5$.
This has the following consequence:
Let $n \in \N$ such that $p^n \ge 3$.
Then $\Hh^i(\cM_g[p^n](\C),\Lambda) = 0$ for all $i > 4g-5$ and all coefficients $\Lambda$.
Here, $\cM_g[p^n]$ is the moduli space of smooth, complex curves of genus $g \ge 2$ with full level-$p^n$ structure;
it parametrizes smooth, complex curves $C$ of genus $g \ge 2$ together with an isomorphism $\Hh^1(C,\Z/p^n\Z) \xrightarrow{\sim} (\Z/p^n\Z)^{2g}$.
The natural morphism $\cM_g[p^n] \to \cM_g$ ``forgetting the level structure'' is a finite \'etale cover and the finite-index subgroup of $\Gamma_g$ corresponding to this cover is torsion-free.

In this paper, we are interested in the groups $\Hh^i(\cM_g[p^n](\C),\F_p)$.
Since $\dim_{\Q}\Hh^i(\cM_g[p^n](\C),\Q) \le \dim_{\F_p}\Hh^i(\cM_g[p^n](\C),\F_p)$, the Cartan--Leray spectral sequence combined with the vanishing of $\Hh^i(\cM_g[p^n](\C),\F_p)$ in degrees $i>4g-5$ implies Theorem~\ref{thm:Harer}.
Our main result is a vanishing statement ``at infinite level'' in this direction, proved without geometric topology, for $\cM_g[p^n]$ and some (partial) compactifications thereof, which we recall in detail in \S~\ref{sect:moduli-curves}.
\begin{restatable}{mainthm}{vanishing}\label{thm:vanishing}
  Let $g \ge 2$ and $p$ be a prime.
  Let $\cM[p^n]$ be one of the following:
  \begin{enumerate}[label={\upshape(\roman*)}]
    \item\label{thm:vanishing-Mg} the moduli space $\cM_g[p^n]$ of smooth curves of genus $g$ over $\C$ with full level-$p^n$ structure,
    \item\label{thm:vanishing-Mgc} the moduli space $\cM^{\cc}_g[p^n]$ of curves of compact type of genus $g$ over $\C$ with full level-$p^n$ structure, or
    \item\label{thm:vanishing-Mgbar} the moduli space $\overline{\cM_g}[p^n]$ of pre-level-$p^n$ curves of genus $g$ over $\C$ with full level-$p^n$ structure.
  \end{enumerate}
  Then we have
  \[ \colim_n\Hh^i_{\et}(\cM[p^n],\F_p) = 0 \]
  for all $i > 4g-5$ in case \ref{thm:vanishing-Mg} and for all $i > \floor*{\frac{7g}{2}} - 4$ in cases \ref{thm:vanishing-Mgc} and \ref{thm:vanishing-Mgbar}.
\end{restatable}
The stack $\overline{\cM_g}[p^n]$ is a smooth, modular compactification of $\cM_g[p^n]$ that was first introduced in \cite{MR2007376};
when $n=0$, a pre-level-$p^n$ curve is simply a stable curve.
A curve of compact type is a stable curve whose dual graph is a tree.
\begin{rem}
  While our bound in case \ref{thm:vanishing-Mg} is the same as the one in Theorem~\ref{thm:Harer}, the only known upper bound on the virtual cohomological dimension (and hence vanishing at finite level) for $\cM^{\cc}_g$ is $5g-6$ \cite[Cor.~0.4]{MR2349902} and thus significantly higher than our bound $\floor*{\frac{7g}{2}} - 4$ for vanishing at infinite level.
  It may be worthwile to further investigate this discrepancy.
\end{rem}
\begin{rem}
  The groups $\colim_n\Hh^i_{\et}(\cM_g[p^n],\F_p)$ are the mod $p$ analogs of the completed cohomology groups for the tower of spaces
  \[ \dotsb \to \cM_g[p^n] \to \dotsb \to \cM_g[p] \to \cM_g; \]
  cf.\ \cite{MR2207783,MR2905536}.
  A d\'evissage argument shows that $\colim_n\Hh^i_{\et}(\cM_g[p^n],\Z/p^s\Z) = 0$ in degrees $i > 4g-5$ for all $s \in \N$ and thus the vanishing of the integral completed cohomology groups
  \[ \lim_s\colim_n\Hh^i_{\et}(\cM_g[p^n],\Z/p^s\Z) \]
  in this range of degrees.
\end{rem}

\subsection*{Strategy of the proof of Theorem~\ref{thm:vanishing}}
We reduce \ref{thm:vanishing-Mg} and \ref{thm:vanishing-Mgbar} to statements about $\cM^{\cc}_g$.
To do so, first observe that since the boundaries $\cM^{\cc}_g[p^n] \smallsetminus \cM_g[p^n]$ are Cartier divisors, the inclusions $j_n \colon \cM_g[p^n] \hookrightarrow \cM^{\cc}_g[p^n]$ are affine morphisms.
In particular, the functors $\R j_{n,*}$ are right t-exact for the perverse t-structure; see \S~\ref{sect:perverse-vanishing} for a review of the necessary prerequisites concerning perverse t-structures.
Parts \ref{thm:vanishing-Mg} and \ref{thm:vanishing-Mgc} of Theorem~\ref{thm:vanishing} thus follow from parts \ref{thm:perverse-perverse} and \ref{thm:perverse-constructible}, respectively, of the following, more general statement.
\begin{restatable}{mainthm}{perverse}\label{thm:perverse}
  Let $g \ge 2$ and $p$ be a prime.
  Let $\cM^{\cc}_g[p^n]$ be the moduli space of curves of compact type of genus $g$ over $\C$ with full level-$p^n$ structure.
  Let $\pi_n \colon \cM^{\cc}_g[p^n] \to \cM^{\cc}_g$ be the maps ``forgetting the level structure.''
  \begin{enumerate}[label={\upshape(\roman*)}]
    \item\label{thm:perverse-perverse} If $K \in \pD(\cM^{\cc}_g,\F_p)^{\le 0}$, we have $\colim_n\Hh^i_{\et}(\cM^{\cc}_g[p^n],\pi^*_n K) = 0$ for all $i > g-2$.
    \item\label{thm:perverse-constructible} If $F$ is a constructible sheaf of $\F_p$-modules on $\cM^{\cc}_g$, we have $\colim_n\Hh^i_{\et}(\cM^{\cc}_g[p^n],\pi^*_n F) = 0$ for all $i > \floor*{\frac{7g}{2}} - 4$.
  \end{enumerate}
\end{restatable}
Here, $\D(\cM^{\cc}_g,\F_p)$ denotes the bounded derived category of constructible sheaves of $\F_p$-modules on $\cM^{\cc}_g$.
Theorem~\ref{thm:perverse} is shown in three steps.
\begin{enumerate}
  \item \emph{Analysis of the Torelli morphism}.\label{summary-Torelli}
    Let $\cA_g$ be the moduli space of principally polarized abelian varieties of dimension $g$ over $\C$.
    The Torelli morphism $t_g \colon \cM^{\cc}_g \to \cA_g$ sends a curve of compact type to the product of the Jacobians of its irreducible components.
    Analyzing the fiber dimensions of $t_g$, we produce bounds on its cohomological amplitude: in case \ref{thm:perverse-perverse}, $\R t_{g,*}K \in \pD(\cA_g,\F_p)^{\le g-2}$ and in case \ref{thm:perverse-constructible}, $\R t_{g,*}(F[3g-3]) \in \pD(\cA_g,\F_p)^{\le \floor*{\frac{g}{2}}-1}$.
    The same bounds hold in the presence of level structures.
    See \S~\ref{subsect:Torelli} for details.
  \item \emph{Passage to toroidal compactifications}.\label{summary-toroidal}
    Let $\rho_n \colon \cA_g[p^n] \to \cA_g$ be the natural covering maps.
    By Step~\ref{summary-Torelli}, it remains to prove that the functor
    \[ \D(\cA_g,\F_p) \to \D(\F_p), \quad K \mapsto \colim_n \R\Gamma(\cA_g[p^n],\rho^*_nK) \]
    takes $\pD(\cA_g,\F_p)^{\le m}$ to $\D(\F_p)^{\le m}$ for all $m \in \Z$.
    After fixing some additional auxiliary data, there are compatible toroidal compactifications $\cA_g[p^n] \subset \overline{\cA_g}[p^n]$ with forgetful maps $\bar{\rho}_{mn} \colon \overline{\cA_g}[p^m] \to \overline{\cA_g}[p^n]$.
    As in the case of $\cM_g[p^n]$, the boundaries are Cartier divisors, and we can reduce the statement to showing that
    \[ \D(\overline{\cA_g}[p^n],\F_p) \to \D(\F_p), \quad L \mapsto \colim_m \R\Gamma(\overline{\cA_g}[p^m],\bar{\rho}^*_{mn}L) \]
    takes $\pD(\overline{\cA_g}[p^n],\F_p)^{\le \ell}$ to $\D(\F_p)^{\le \ell}$ for all $n \in \Z_{\ge 0}$ and $\ell \in \Z$.
  \item \emph{Use of perfectoid covers}.\label{summary-perfectoid}
    By work of Scholze and Pilloni--Stroh, the inverse limit of the projective system $\overline{\cA_g}[p^n]$, $n \in \Z_{\ge 0}$, is similar (in the sense of \cite[Def.~2.4.1]{MR3272049}) to a perfectoid space; see Example~\ref{exmp:Ag-tor}.
    Using his ``primitive comparison theorem,'' Scholze \cite[\S~4.2]{MR3418533} deduced from this that $\colim_n\Hh^i(\overline{\cA_g}[p^n],\F_p) = 0$ for $i > \dim \cA_g$.
    In \S~\ref{sect:perfectoid-Artin}, we prove a more general statement for Zariski-constructible sheaves on cofiltered inverse systems of proper rigid spaces with finite transition maps whose inverse limit is similar to a perfectoid space;
    this yields the claim about perverse t-structures from Step~\ref{summary-toroidal}.
    After a standard d\'evissage, the key new case is the extension by zero of a non-trivial local system defined on a Zariski-open subset (Lemma~\ref{lem:perfectoid-Artin-key}), whose proof relies on \cite{bhatt2019prisms}.
\end{enumerate}

In order to reduce Theorem~\ref{thm:vanishing}.\ref{thm:vanishing-Mgbar} to a statement about $\cM^{\cc}_g$, we show the following assertion, which holds over any algebraically closed field $k$ of characteristic not equal to $p$.
\begin{restatable}{mainthm}{Mgbar}\label{thm:Mgbar}
  Let $\Lambda_0$ be an \'etale $\F_p$-local system on $\overline{\cM_g}$, with pullbacks $\Lambda_n$ to $\overline{\cM_g}[p^n]$.
  Then for all $i \ge 0$, the natural map
  \[ \colim_n \Hh^i_{\et}(\overline{\cM_g}[p^n],\Lambda_n) \to \colim_n \Hh^i_{\et}(\cM^{\cc}_g[p^n],\Lambda_n) \]
  is an isomorphism.
\end{restatable}
The key observation in the proof is that the transition maps $\overline{\cM_g}[p^{n+1}] \to \overline{\cM_g}[p^n]$ are highly ramified over the boundary $\overline{\cM_g}[p^n] \smallsetminus \cM^{\cc}_g[p^n]$:
on complete local rings, they are given by 
\[ k\llbracket t_1,\dotsc,t_{3g-3}\rrbracket \to k\llbracket t_1,\dotsc,t_{3g-3}\rrbracket, \quad t_i \mapsto \begin{cases} t^p_i & \text{if } 1 \le i \le \gamma, \\ t_i & \text{if } \gamma < i < 3g-3 \end{cases}  \]
after a suitable choice of coordinates $t_i$ for which $\overline{\cM_g}[p^n] \smallsetminus \cM^{\cc}_g[p^n]$ corresponds to $\{ t_1 \dotsm t_\gamma = 0 \}$;
see Lemma~\ref{lem:Mg-p-ramified}.
The claim then follows from a careful analysis of the maps between the excision triangles for the inclusions $\cM^{\cc}_g[p^n] \subset \overline{\cM_g}[p^n]$.
We refer to \S~\ref{sect:boundary-vanishing} for details and a more general statement about projective systems of smooth Deligne--Mumford stacks with similar ramification over a system of normal crossings divisors.

In the argument, it really seems necessary to use the compactifictions via pre-level-$p^n$ curves from \cite{MR2007376}.
A ``na\"{i}ve'' compactification of $M_g[p^n]$ as the normalization of $\overline{M_g}$ in the function field of $M_g[p^n]$, which turns out to be the coarse moduli space of $\overline{\cM_g}[p^n]$, does not exhibit the same ramification behavior; see Example~\ref{exmp:Mumford-compactification}.
For the toroidal compactifications $\cA_g[p^n] \subset \overline{\cA_g}[p^n]$, such complications do not arise:
the $\overline{\cA_g}[p^n]$ are algebraic spaces for $p^n \geq 3$ and the maps $\overline{\cA_g}[p^{n+1}] \to \overline{\cA_g}[p^n]$ are sufficiently ramified over the boundary $\overline{\cA_g}[p^n] \smallsetminus \cA_g[p^n]$ so that
\[ \colim_n \Hh^i_{\et}(\overline{\cA_g}[p^n],\F_p) \to \colim_n \Hh^i_{\et}(\cA_g[p^n],\F_p) \]
is again an isomorphism.
Thus, the aforementioned results of Scholze \cite{MR3418533} and Pilloni--Stroh \cite[Thm.~0.4]{MR3512528} yield the analogous statement $\colim_n \Hh^i_{\et}(\cA_g[p^n],\F_p) = 0$ for $i > \dim \cA_g = \frac{g(g+1)}{2}$, as had been observed earlier by Scholze.

\subsection*{Notation}
Throughout this paper, we fix a prime $p \in \N$.
For any Deligne--Mumford stack $\cX$, we denote from now on the bounded derived category of constructible sheaves of $\F_p$-modules on the small \'etale site of $\cX$ by $\D(\cX)$;
for the algebraic stacks appearing in \S~\ref{sect:perverse-vanishing}, we use the lisse-\'etale site as developed in \cite{MR2312554} or \cite{liu2017enhanced}.

\section{Vanishing from perverse sheaves}\label{sect:perverse-vanishing}

Throughout this section, fix a separably closed field $k$ of characteristic not equal to $p$.
Recall the following definition from \cite[\S~4]{MR2480756}.
\begin{defn}\label{defn:perverse-t-structure}
  Let $\cX$ be an algebraic stack of finite type over $k$.
  Let $\pi \colon X \to \cX$ be a smooth surjection from a scheme $X$ of finite type over $k$.
  Denote the \emph{relative dimension} of $\pi$ at a point $x \in X$ by $d_\pi(x) \colonequals \dim_x(X_{\pi(x)})$ and set $\dim(x) \colonequals \trdeg(\kappa(x)/k)$.
  For any $n \in \Z$, the category of \emph{complexes bounded above $n$ for the perverse t-structure} is the full subcategory
  \[ \pD(\cX)^{\le n} \colonequals \{ K \in \D(\cX) \suchthat \cH^j\bigr((\pi^*K)_x\bigl) = 0 \text{ for all } x \in X \text{ and } j > n + d_\pi(x) - \dim(x) \} \subset \D(\cX), \]
  where $(\pi^*K)_x$ is the derived pullback of $\pi^*K$ under the inclusion of $x$ into $X$,
\end{defn}
By \cite[Lem.~4.1]{MR2480756}, Definition~\ref{defn:perverse-t-structure} is independent of the choice of $\pi \colon X \to \cX$.
Moreover, since the restriction functor induces an equivalence between the categories of constructible sheaves on $X_{\liset}$ and $X_{\et}$, we identify $\pi^*K$ with a sheaf on the small \'etale site of $X$.

In this section, we collect some known foundational results about the derived direct image of $\pD(\cX)^{\le n}$ under certain morphisms for later reference.
Although not strictly necessary, it will be convenient later on to work in the generality of stacks.
We use the dimension theory for stacks developed in \cite[\href{https://stacks.math.columbia.edu/tag/0DRE}{\S~0DRE}]{stacks-project}.
\begin{rem}\label{rem:constant-fiber-dimension}
  The relative dimension of any smooth morphism is locally constant on the domain \cite[\href{https://stacks.math.columbia.edu/tag/0DRQ}{Lem.~0DRQ}]{stacks-project}.
  In order to show that a given $K \in \D(\cX)$ is contained in $\pD(\cX)^{\le n}$, it therefore suffices to check that $\pi^*K \in \pD(X)^{\le n + d_\pi}$ for every smooth (not necessarily surjective) morphism $\pi \colon X \to \cX$ from a connected scheme $X$ of finite type with constant relative dimension $d_\pi$ at every $x \in X$.
\end{rem}
\begin{thm}\label{thm:affine-exact}
  Let $f \colon \cX \to \cY$ be an affine morphism between algebraic stacks of finite type over $k$.
  Then $\R f_*\bigl(\pD(\cX)^{\le n}\bigr) \subseteq \pD(\cY)^{\le n}$ for all $n \in \Z$.
\end{thm}
\begin{proof}
  Let $K \in \pD(\cX)^{\le n}$.
  By Remark~\ref{rem:constant-fiber-dimension}, it suffices to check that $\pi^*\R f_*K \in \pD(Y)^{\le n + d_\pi}$ for every smooth morphism $\pi \colon Y \to \cY$ from a connected scheme $Y$ of finite type over $k$ of relative dimension $d_\pi$.
  Since $f$ is affine, $X \colonequals \cX \times_{\cY} Y$ is a scheme as well.
  Denote the projection to the first and second factor by $\pi'$ and $f'$, respectively.
  By \cite[\href{https://stacks.math.columbia.edu/tag/0DRN}{Lem.~0DRN}]{stacks-project}, $d_{\pi'}$ is constant and equals $d_\pi$.
  In particular, $\pi^{\prime,*}K \in \pD(X)^{\le n + d_\pi}$.
  By \cite[Prop.~9.8.(i)]{MR2312554}, $\pi^*\R f_*K \simeq \R f'_*\pi^{\prime,*}K$ as complexes on $Y_{\et}$, so the assertion follows from \cite[Thm.~XIV.3.1]{SGA4}.
\end{proof}
We will mainly use Theorem~\ref{thm:affine-exact} when $f$ is the inclusion of the complement of a Cartier divisor.
\begin{prop}\label{prop:proper-exact}
  Let $f \colon \cX \to \cY$ be a proper morphism between algebraic stacks of finite type over $k$ which is representable by algebraic spaces.
  Assume that for any $y \in \abs{\cY}$, the fiber dimension $\dim(\cX_y)$ is at most $d$.
  Then $\R f_*\bigl(\pD(\cX)^{\le n}\bigr) \subseteq \pD(\cY)^{\le n+d}$ for all $n \in \Z$.
\end{prop}
\begin{proof}
  As in the proof of Theorem~\ref{thm:affine-exact}, choose a smooth morphism $\pi \colon Y \to \cY$ from a connected scheme of finite type over $k$ with relative dimension $d_\pi$ at all $y \in Y$.
  The fiber product $X \colonequals \cX \times_{\cY} Y$ is an algebraic space, again of constant relative dimension $d_\pi$ over $\cX$. 
  By \cite[Prop.~9.8.(i)]{MR2312554}, it therefore suffices to prove the statement in case $\cY$ is a scheme and $\cX$ an algebraic space.

  We proceed by induction on $\dim \cX$, the case $\dim \cX = 0$ being trivial.
  For the inductive step, since the assertion is local on $\cY$, we may assume that $\cY$, and hence $\cX$, is separated.
  In particular, $\cX$ has dense, open schematic locus;
  we can choose a dense, open subscheme $U \subseteq \cY$ such that $f^{-1}(U) \subset \cX$ is represented by a scheme because $f$ is closed.

  Denote by $j \colon U \hookrightarrow \cY$ and $j' \colon f^{-1}(U) \hookrightarrow \cX$ the open immersions and by $i$ and $i'$ the closed immersions of their complements with their reduced subspace structures, respectively.
  For any $K \in \pD(\cX)^{\le n}$, the induction hypothesis ensures $i_*i^*\R f_*K \simeq i_*\R f_*i^{\prime,*}K \in \pD^{\le n+d}(\cY)$.
  On the other hand, $j_!j^*\R f_*K \simeq j_!\R f_*j^{\prime,*}K \in \pD^{\le n+d}(\cY)$ by \cite[\S~4.2.4]{MR751966}.
  The assertion then follows from the long exact excision sequence for $\R f_*K$.
\end{proof}
Under stricter assumptions on the allowable loci with a given fiber dimension, more can be said about the direct image of $\F_p[\dim \cX]$.
\begin{defn}\label{defn:defect}
  Let $f \colon \cX \to \cY$ be a proper morphism between algebraic stacks of finite type over $k$ which is representable by algebraic spaces.
  Let $\abs{\cY} = \bigsqcup_{i \in I} S_i$ be a finite stratification of $\abs{\cY}$ into locally closed subspaces $S_i$.
  Let $\cY_i$ be the corresponding locally closed substacks with their reduced induced substack structure (cf.\ \cite[\href{https://stacks.math.columbia.edu/tag/06FK}{Remark~06FK}]{stacks-project}).
  Assume that $f$ has constant relative dimension $d_i$ over $\cY_i$.
  Then the \emph{defect of semi-smallness of $f$} is
  \[ r(f) \colonequals \max_{i \in I} \{ 2d_i + \dim \cY_i - \dim \cX \}. \]
\end{defn} 
\begin{exmp}
  If $f$ is the blowup of a closed subvariety of codimension $c \ge 2$ in a variety $Y$, we have $r(f) = c-2$.
\end{exmp}
\begin{lem}\label{lem:defect-exact}
  Let $f \colon \cX \to \cY$ and $\abs{\cY} = \bigsqcup_{i \in I} S_i$ be as in Definition~\ref{defn:defect}.
  Let $F$ be a constructible sheaf of $\F_p$-modules on $\cX$.
  Then $\R f_*(F[\dim \cX]) \in \pD(\cY)^{\le r(f)}$.
\end{lem}
\begin{proof}
  Choose a smooth surjection $Y \to \cY$ from a scheme $Y$ of finite type over $k$, giving rise to a fiber square
  \[ \begin{tikzcd}
      X \colonequals \cX \times_\cY Y \arrow[r,"f'"] \arrow[d,"\pi'"] & Y \arrow[d,"\pi"] \\
      \cX \arrow[r,"f"] & \cY.
  \end{tikzcd} \]
  Let $y \in Y$.
  Then
  \[ \bigl(\pi^*\R f_*(F[\dim\cX])\bigr)_y \simeq \bigl(\R f'_*\pi^{\prime,*}(F[\dim\cX])\bigr)_y \simeq \R\Gamma(X_y,\pi^{\prime,*}F[\dim\cX]) \]
  by proper base change.
  If $\pi(y) \in S_i$, we have $\dim X_y = d_i$ and hence $\cH^j\bigl(\bigl(\pi^*\R f_*(F[\dim\cX])\bigr)_y\bigr) \simeq \Hh^{j+\dim\cX}(X_y,\pi^{\prime,*}F) = 0$ for all $j > 2d_i - \dim\cX$.
  The assertion thus follows from the inequality
  \[ 2d_i - \dim\cX \le r(f) - \dim\cY_i \le r(f) + d_\pi(y) - \dim(y). \qedhere \]
\end{proof}
\begin{rem}
  In fact, a combination of Lemma~\ref{lem:defect-exact} and Proposition~\ref{prop:proper-exact} shows that $\R f_*(F[\dim\cX]) \in \pD(\cY)^{\le r}$, where $r \colonequals \min \{ r(f), 2 \dim \supp F - \dim\cX \}$.
  On the other hand, the bound of Lemma~\ref{lem:defect-exact} is sharp in general.
  For example, choose $i \in I$ such that $\dim \cY_i = r(f) + \dim\cX - 2d_i$.
  Let $y \in \pi^{-1}(S_i)$ be the generic point of an irreducible component such that $\dim \cY_i = \dim_{\pi(y)}\cY_i$.
  Then equality holds at every step of the last inequality in the proof and thus $\R f_*(\F_p[\dim \cX]) \notin \pD(\cY)^{\le r(f)-1}$.
\end{rem}

\section{Perfectoid Artin vanishing}\label{sect:perfectoid-Artin}

The results of this section are inspired by the following statement due to Artin and Grothendieck.
\begin{thm}[{\cite[Cor.~XIV.3.2]{SGA4}}]\label{thm:Artin-vanishing}
  Suppose $X$ is an affine algebraic variety over a separably closed field.
  Let $F$ be a torsion abelian \'etale sheaf on $X$.
  Then $\Hh^n_{\et}(X,F) = 0$ for all $n > \dim X$.
\end{thm}
When $X$ is proper, Theorem~\ref{thm:Artin-vanishing} fails; nonetheless, we can prove a similar version for certain inverse systems of varieties.
From here on, fix a complete, algebraically closed extension $C$ of $\Q_p$.
Let $\cO_C$ be its ring of integers and $\varpi$ be a pseudouniformizer.
We will work in the following setting.
\begin{sit}\label{sit:inverse-system}
  Let $I$ be a cofiltered category with final object $0 \in I$.
  Let $X_i$, $i \in I$, be a cofiltered inverse system of quasi-compact and quasi-separated rigid spaces over $C$ of dimension $d$ with finite transition maps $\pi_{ji} \colon X_j \to X_i$.
  Assume that the diamond $\lim_i X^\lozenge_i$ is representable by a perfectoid space $X$.
\end{sit}
In order to have a good theory of \'etale cohomology, we will always identify rigid spaces with their associated adic spaces \cite[Prop.~4.3]{MR1306024}.
In the setting of this paper, where we work exclusively over the fixed perfectoid ground field $C$, a diamond is, in short, a pro-\'etale sheaf on the category of perfectoid spaces over $C$ which is a quotient of a perfectoid space by a pro-\'etale equivalence relation (akin to the definition of algebraic spaces).
There is a ``diamondification'' functor which associates to any analytic adic space $Y$ over $C$ a diamond $Y^\lozenge$.
We refer to \cite{scholze2018berkeley} for an introduction to these ideas and to \cite{scholze2017etale} for a comprehensive reference.
Later in this paper, we will mostly be interested in the special case where the $X_i$ are algebraic varieties and there exists a perfectoid space $X$ such that $X \sim \lim_i X_i$ in the sense of \cite[Def.~2.4.1]{MR3272049}.
\begin{thm}[Perfectoid Artin vanishing]\label{thm:perfectoid-Artin}
  In Situation~\ref{sit:inverse-system}, let $F_0$ be a Zariski-constructible sheaf of $\F_p$-modules on $X_0$ and $F_i \colonequals \pi^*_{i0}F_0$ for all $i \in I$.
  Assume that $X_0$ is proper.
  Then for all $n > d$
  \[ \colim_i \Hh^n_{\et}(X_i,F_i) = 0. \]
\end{thm}
An \'etale sheaf $F_0$ of $\F_p$-modules on $X_0$ is Zariski-constructible if there is a finite stratification $X_0 = \bigsqcup X_{0,\alpha}$ such that each stratum is locally closed in the Zariski topology of $X_0$ and the restrictions $\restr{F_0}{X_{0,\alpha}}$ are finite locally constant.
For example, every constructible sheaf of $\F_p$-modules on a scheme of finite type over $C$ (in the usual algebro-geometric sense) gives rise to a Zariski-constructible sheaf on its analytification.
We refer to \cite{hansen-vanishing} for a detailed account in the general rigid-analytic setting.
When $F_0 = j_{0,!}\F_p$ for some dense, Zariski-open $j_0 \colon U_0 \hookrightarrow X_0$, Theorem~\ref{thm:perfectoid-Artin} was proven in \cite[\S~4.2]{MR3418533}.

\subsection{Examples and counterexamples}

Before delving into the proof of Theorem~\ref{thm:perfectoid-Artin}, we give some examples and applications.
\begin{exmp}\label{exmp:proj-perfectoid}
  When $X_0 \subseteq \PP^N_C$ is a projective variety, we obtain a natural inverse system with the required properties by putting $X_i \colonequals X_0 \times_{\PP^N_C,\Phi^i} \PP^N_C$, $i \in \Z_{\ge 0}$, where $\Phi$ is the ``mock Frobenius''
  \[ \Phi \colon \PP^N_C \to \PP^N_C, \quad [x_0:\dotsb:x_N] \mapsto [x^p_0:\dotsb:x^p_N]. \]
  In this setting, Esnault \cite[Thm.~5.1]{esnault2018cohomological} gave a proof of Theorem~\ref{thm:perfectoid-Artin} that does not require perfectoid techniques.
  In fact, her argument establishes this special case over any algebraically closed field of characteristic not equal to $p$.
  It seems interesting to investigate whether Theorem~\ref{thm:perfectoid-Artin} holds true over more general algebraically closed, non-archimedean fields as well.
\end{exmp}
\begin{exmp}
  Let $A$ be an abelian variety, or more generally an abeloid variety, over $C$.
  Consider the cofiltered inverse system
  \[ \dotsb \xrightarrow{[p]} A \xrightarrow{[p]} A \xrightarrow{[p]} A \]
  in which the transition maps are multiplication by $p$.
  In \cite{blakestad2018perfectoid}, the authors construct a perfectoid space $A_\infty$ with $A_\infty \sim \lim_{[p]} A$ and thus $A^\lozenge_\infty \simeq \lim_{[p]} A^\lozenge$.
  In particular, this produces interesting non-algebraic examples.
\end{exmp}
\begin{exmp}\label{exmp:Ag-tor}
  Fix, once and for all, a smooth, $\GL_g(\Z)$-admissible polyhedral decomposition of the cone of positive semi-definite quadratic forms on $\R^g$ whose null space is defined over $\Q$; see \cite[Def.~IV.2.2, IV.2.3]{MR1083353}.
  Let $\overline{\cA_g}[m]$ be the toroidal compactification of the moduli space $\cA_g[m]$ of principally polarized abelian varieties of dimension $g$ over $C$ with full level-$m$ structure which is determined by this decomposition.
  By \cite[Thm.~IV.6.7]{MR1083353}, $\overline{\cA_g}[m]$ is a smooth and proper algebraic stack in which $\overline{\cA_g}[m] \smallsetminus \cA_g[m]$ is a normal crossings divisor (see Definition~\ref{defn:nc-divisor}).
  If $m \ge 3$, it is an algebraic space \cite[Cor.~IV.6.9]{MR1083353}, and even a projective variety under certain convexity conditions on the decomposition \cite[\S~V.5]{MR1083353}.
  
  We obtain a cofiltered inverse system $X_i \colonequals \overline{\cA_g}[p^i]$, $i \in \Z_{\ge 0}$, of projective varieties with finite transition maps ``forgetting level structure.''
  Work of Scholze \cite{MR3418533} and Pilloni--Stroh \cite[Thm.~0.4]{MR3512528} constructs a perfectoid space $X \sim \lim_i X^{\ad}_i$.
  Thus, Theorem~\ref{thm:perfectoid-Artin} applies to this inverse system.
\end{exmp}
For later purposes, it will be beneficial to formulate Theorem~\ref{thm:perfectoid-Artin} for arbitrary systems of semiperverse sheaves.
We only consider the case where the $X_i$ are (the analytifications of) proper schemes over $C$, in which we can use the notation set up in \S~\ref{sect:perverse-vanishing}.
\begin{cor}\label{cor:perfectoid-Artin-perverse}
  Let $m \in \Z$ and $K_i \in \pD(X_i)^{\le m}$ for all $i \in I$.
  Assume that for every morphism $i \to j$ of $I$, there exists $\varphi_{ij} \colon \pi^*_{ij}K_j \to K_i$ such that $\varphi_{ik} = \varphi_{ij} \circ \pi^*_{ij}\varphi_{jk}$ for all $i \to j \to k$.
  Then
  \[ \colim_i \Hh^n_{\et}(X_i,K_i) = 0 \]
  for all $n > m$.
\end{cor}
\begin{proof}
  By Lemma~\ref{lem:staircase-colimit} below applied to the $\Arr(I^{\op})$-shaped diagram $H_{j \to i} \colonequals \Hh^n_{\et}(X_i,\pi^*_{ij}K_j)$ with transition maps induced by the natural pullback morphisms and the $\varphi_{ij}$, we may assume that $K_i \simeq \pi^*_{i0}K_0$ for all $i \in I$.
  Consider the hypercohomology spectral sequences
  \[ \Hh^r_{\et}\bigl(X_i,\cH^s(K_i)\bigr) \Longrightarrow \Hh^{r+s}_{\et}(X_i,K_i), \]
  which are compatible in $i \in I$.
  By the ``colimit lemma'' (cf.\ e.g.\ \cite[Prop.~3.3]{MR1466977}), it suffices to check that
  \[ \colim_i \Hh^r_{\et}\bigl(X_i,\cH^s(K_i)\bigr) = 0 \]
  whenever $r+s > m$.
  By Theorem~\ref{thm:perfectoid-Artin} and Proposition~\ref{prop:perfectoidization} below, the colimit is $0$ when $r > \dim\supp\cH^s(K_0)$, so the statement follows from the semiperversity condition $\dim\supp\cH^s(K_0) \le m-s$.
\end{proof}
It remains to prove Lemma~\ref{lem:staircase-colimit}.
Recall that the arrow category of a category $J$ is the category $\Fun(\{0 \to 1\},J)$ of functors from the interval category $\{0 \to 1\}$ to $J$ whose objects are the morphisms of $J$.
In \S~\ref{sect:moduli-curves-vanishing}, we will take $J = \Z_{\ge 0}$;
in that case, $\Arr(J)$ can be identified with the ``staircase'' diagram $\bigl\{ (j,k) \in (\Z_{\ge 0})^2 \suchthat j \le k \bigr\}$, considered as a partially ordered set via the product order.
\begin{lem}\label{lem:staircase-colimit}
  Let $\cC$ be a category that admits all filtered colimits.
  Let $J$ be a filtered category with associated arrow category $\Arr(J)$ and $H \colon \Arr(J) \to \cC$ be an $\Arr(J)$-shaped diagram of $\cC$.
  Then
  \[ \colim_{j \in J} H_{(\id \colon j \to j)} \simeq \colim_{j \in J} \colim_{(j \to k) \in \Arr(J)} H_{(j \to k)}. \]
\end{lem}
\begin{proof}
  Let $F \colon \Arr(J) \to J,\, (j \to k) \mapsto j$ be the projection onto the source and $G \colon J \to *$ be the functor to the terminal category.
  Since the fully faithful ``diagonal'' subcategory $J \subset \Arr(J)$ consisting of all identity morphisms $(\id \colon j \to j)$ is cofinal, we have $\colim_{j \in J} H_{(\id \colon j \to j)} \simeq \colim_{\Arr(J)} H$.
  However, we can compute $\colim_{\Arr(J)} H$ another way: it is the left Kan extension $\Lan_{G \circ F} H$ of $H$ along $G \circ F$, which is naturally isomorphic to $\Lan_G \Lan_F H$.
  By the universal property of left Kan extensions, the functor $\Lan_F H$ is the $J$-shaped diagram $\bigl(\colim_{(j \to k) \in \Arr(J)} H_{(j \to k)}\bigr)_{j \in J}$ of colimits with fixed source, hence
  \[ \Lan_G \Lan_F H \simeq \Lan_G \Bigl(\colim_{(j \to k) \in \Arr(J)} H_{(j \to k)}\Bigr)_{j \in J} \simeq \colim_{j \in J} \colim_{(j \to k) \in \Arr(J)} H_{(j \to k)}. \qedhere \]
\end{proof}
We conclude with two further examples to indicate the limits of our methods.
\begin{exmp}\label{exmp:no-vanishing-generically-finite}
  We work in Situation~\ref{sit:inverse-system} and assume that $X_0$ is proper.
  If $f_0 \colon Y_0 \to X_0$ is a finite morphism and $Y_i \colonequals Y_0 \times_{X_0} X_i$, Theorem~\ref{thm:perfectoid-Artin} applied to $f_{0,*}\F_p$ on $X_0$ (which is Zariski-constructible by \cite[Prop.~2.3]{hansen-vanishing}) shows that $\colim_i \Hh^n_{\et}(Y_i,\F_p) = 0$ for all $n > d$.
  In fact, this will be established independently in Lemma~\ref{lem:perfectoid-Artin-key} as a key step toward the proof of Theorem~\ref{thm:perfectoid-Artin}.
  When $f_0$ is only required to be generically finite, the statement is already false for the blowup of $\PP^d_C$ in a point.
  
  More precisely, let $X_i \colonequals \PP^d_C$, $i \in \Z_{\ge 0}$, be the inverse system whose transition maps are the mock Frobenii from Example~\ref{exmp:proj-perfectoid}.
  Let $Q \colonequals [1:1:\dotsc:1]$ and $Y_0 \colonequals \Bl_Q \PP^d_C \xrightarrow{f_0} \PP^d_C$.
  Since blowing up commutes with flat base change \cite[\href{https://stacks.math.columbia.edu/tag/0805}{Lem.~0805}]{stacks-project}, $Y_i$ is the blowup of $\PP^d_C$ in the points $Q_{i1},\dotsc,Q_{ip^{id}}$ whose coordinates are all $p^i$-th roots of unity.
  Let $E$ be the exceptional divisor in $Y_0$ and $E_{i1},\dotsc,E_{ip^{id}}$ the exceptional divisors in $Y_i$.

  On cohomology, we have
  \[ \Hh^{2d-2}_{\et}(Y_i,\F_p) = \Hh^{2d-2}_{\et}(\PP^d_C,\F_p) \oplus \bigoplus_j \Hh^{2d-4}_{\et}(E_{ij},\F_p), \]
  where the summands coming from the $E_{ij}$ are generated by the pushforwards of $\cc_1\bigl(\cO_{E_{ij}}(1)\bigr)^{d-2}$ along the inclusions $E_{ij} \hookrightarrow X_i$.
  Moreover, the pullback of $\cc_1\bigl(\cO_E(1)\bigr)$ along $Y_i \to Y_0$ is $\sum_j \cc_1\bigl(\cO_{E_{ij}}(1)\bigr)$ and different Chern classes in the sum intersect to $0$, so $\cc_1\bigl(\cO_E(1)\bigr)^{d-2}$ pulls back to $\sum_j \cc_1\bigl(\cO_{E_{ij}}(1)\bigr)^{d-2}$.
  In other words, the transition map
  \[ \Hh^{2d-2}_{\et}(\PP^d_C,\F_p) \oplus \Hh^{2d-4}_{\et}(E,\F_p) \to \Hh^{2d-2}(\PP^d_C,\F_p) \oplus \bigoplus_j \Hh^{2d-4}_{\et}(E_{ij},\F_p) \]
  is multiplication by $p^i$ on the summand $H^{2d-2}_{\et}(\PP^d_C,\F_p)$, but the diagonal map on $\Hh^{2d-4}_{\et}(E,\F_p)$.
  In particular, $\colim_i \Hh^{2d-2}_{\et}(Y_i,\F_p) \neq 0$.
\end{exmp}
The next example shows why Theorem~\ref{thm:perfectoid-Artin} does not hold for more general classes of constructible sheaves which can be found in \cite{MR1734903}.
\begin{exmp}
  Let $X_i \colonequals \PP^1_{\C_p}$, $i \in \Z_{\ge 0}$, with the transition maps from Example~\ref{exmp:proj-perfectoid}.
  Let $j_i \colonequals U_i \hookrightarrow X_i$ be the closed unit disc;
  in particular, $U_i \simeq U_0 \times_{X_0} X_i$.
  We claim that for $F_0 \colonequals j_{0,!}\F_p$, we have $\colim_i \Hh^2_{\et}(X_i,F_i) \neq 0$.

  Let $X = \PP^{1,\perf}_{\C_p}$ and $j \colon U \hookrightarrow X$ be the perfectoid inverse limit of the system $X_i$ and $U_i$, respectively.
  Likewise, let $Y_i \colonequals \PP^1_{\C^\flat_p}$, $i \in \Z_{\ge 0}$, be the inverse system of projective lines over $\C^\flat_p$ with relative Frobenius as transition maps, $k_i \colon V_i \hookrightarrow Y_i$ the closed unit discs, and $Y \simeq \PP^{1,\perf}_{\C^\flat_p}$ and $k \colon V \hookrightarrow Y$ their perfectoid inverse limits.
  The tilting equivalence \cite[Thm.~7.12]{MR3090258} \cite[Cor.~8.3.6]{MR3379653} and \cite[Thm.~7.17, Cor.~7.19]{MR3090258} yield a commutative diagram of topoi
  \[ \begin{tikzcd}
      V^\sim_{0,\et} \arrow[d,hook,"k_0"] & V^\sim_{\et} \arrow[l,"\sim"'] \arrow[r,phantom,"\simeq"] \arrow[d,hook,"k"] &[-1.2em] U^\sim_{\et} \arrow[r,phantom,"\simeq"] \arrow[d,hook,"j"] &[-1.2em] \lim_i U^\sim_{i,\et} \arrow[r] & U^\sim_{0,\et} \arrow[d,hook,"j_0"] \\
      Y^\sim_{0,\et} & Y^\sim_{\et} \arrow[l,"\sim"'] \arrow[r,phantom,"\simeq"] & X^\sim_{\et} \arrow[r,phantom,"\simeq"] & \lim_i X^\sim_{i,\et} \arrow[r] & X^\sim_{0,\et}
  \end{tikzcd} \]
  and thus via base change \cite[Lem.~XVII.5.1.2]{SGA4} the identification
  \[ \colim_i \Hh^2_{\et}(X_i,F_i) \simeq \Hh^2_{\et}(X,j_!\F_p) \simeq \Hh^2_{\et}(Y,k_!\F_p) \simeq \Hh^2_{\et}(Y_0,k_{0,!}\F_p). \]

  Now we can proceed similarly to \cite[\S~0.2]{MR1734903}.
  We have a short exact sequence of \'etale sheaves on $Y_0 = \PP^1_{\C^\flat_p}$
  \begin{equation}\label{eqn:cohomology-disc-projective-line}
    0 \to k_{0,!}\F_p \to \F_p \to \iota_*\F_p \to 0,
  \end{equation}
  where $\iota \colon W^\sim_{\et,Y_0} \hookrightarrow Y^\sim_{0,\et}$ is the inclusion of the closed subtopos complementary to $V^\sim_{0,\et}$ \cite[\S~IV.9.3]{SGA4}, which corresponds to $W \colonequals \{ x \in Y_0 \suchthat \abs{x} > 1 \}$ (in the language of \cite{MR1734903}, the \'etale topos of the pseudo-adic space $(Y_0,W)$).
  By GAGA for rigid \'etale cohomology in the proper case (cf. \cite[Cor.~6.18]{bhatt2018excision} or \cite[Thm.~3.2.10]{MR1734903}), the Artin--Schreier sequence, and \cite[Cor.~X.5.2]{SGA4}, $\Hh^1_{\et}(Y_0,\F_p) \simeq \Hh^2_{\et}(Y_0,\F_p) \simeq 0$.
  Thus, the long exact sequence in cohomology for (\ref{eqn:cohomology-disc-projective-line}) gives an isomorphism
  \[ \Hh^1(W_{\et,Y_0},\F_p) \xrightarrow{\sim} \Hh^2_{\et}(Y_0,k_{0,!}\F_p). \]

  Let $V \colonequals \cO_{\C^\flat_p}$ and $V\{T\}$ be the henselization of $V[T]_{(\fm_V,T)}$.
  Choose a pseudouniformizer $\varpi^\flat$ of $V$.
  As in \cite[Ex.~0.2.5]{MR1734903}, we can conclude from \cite[Thm.~0.2.4/Thm.~3.2.1]{MR1734903} that
  \[ \Hh^1(W_{\et,Y_0},\F_p) \simeq \Hh^1_{\et}\bigl(\Spec V\{T\}\bigl[\tfrac{1}{\varpi^\flat}\bigr],\F_p\bigr). \]
  Next, we show that the \'etale $\F_p$-torsor $V\{T\}\bigl[\tfrac{1}{\varpi^\flat}\bigr] \to \bigl(V\{T\}\bigl[\tfrac{1}{\varpi^\flat}\bigr]\bigr)[S]/\bigl(S^p - S - \tfrac{T}{\varpi^\flat}\bigr)$ is non-trivial;
  this implies $\Hh^1_{\et}\bigl(\Spec V\{T\}\bigl[\tfrac{1}{\varpi^\flat}\bigr],\F_p\bigr) \neq 0$ and hence finishes the proof of the claim.
  
  Note that the map of pairs
  \[ \bigl(V[T]_{(\fm_V,T)},(\varpi^\flat,T)\bigr) \to \bigl(V[T]_{(\fm_V,T)},(\fm_V,T)\bigr) \]
  induces an isomorphism on henselizations by \cite[\href{https://stacks.math.columbia.edu/tag/0F0L}{Lem.~0F0L}]{stacks-project}.
  Therefore, $V[T]_{(\fm_V,T)}$ and $V\{T\}$ have isomorphic $(\varpi^\flat,T)$-adic completions \cite[\href{https://stacks.math.columbia.edu/tag/0AGU}{Lem.~0AGU}]{stacks-project}.
  On the other hand, the $(\varpi^\flat,T)$-adic completion of $V[T]_{(\fm_V,T)}$ is identified via the string of isomorphisms
  \begin{IEEEeqnarray*}{rCl}
    \lim_n V[T]_{(\fm_V,T)}/(\varpi^\flat,T)^n & \simeq & \lim_{a,b} V[T]_{(\fm_V,T)}/(\varpi^{\flat,a},T^b) \simeq \lim_{a,b} \bigl(V[T]/(\varpi^{\flat,a},T^b)\bigr)_{(\fm_V,T)} \\
    & \simeq & \lim_{a,b} V[T]/(\varpi^{\flat,a},T^b) \simeq \lim_b \bigl( \lim_a V/(\varpi^{\flat,a}) \bigr) [T]/(T^b) \\
    & \simeq & \lim_b V[T]/(T^b) \simeq V \llbracket T \rrbracket
  \end{IEEEeqnarray*}
  because the chain of ideals $(\varpi^{\flat,n},T^n)$ is cofinal with $(\varpi^\flat,T)^n$, localization is exact, $(\fm_V,T)$ is maximal and $\rad\bigl((\varpi^{\flat,a},T^b)\bigr) = (\fm_V,T)$, completion commutes with finite sums, and $V$ is $\varpi^\flat$-adically complete, respectively.
  We obtain an induced map $\theta \colon \bigl(V\{T\}\bigl[\tfrac{1}{\varpi^\flat}\bigr]\bigr)[S] \to \bigl(V \llbracket T \rrbracket\bigl[\tfrac{1}{\varpi^\flat}\bigr]\bigr)[S] \subset \bigl(\C^\flat_p \llbracket T \rrbracket\bigr)[S]$.

  Assume that $S^p - S - \tfrac{T}{\varpi^\flat} = g \cdot h$ in $\bigl(V\{T\}\bigl[\tfrac{1}{\varpi^\flat}\bigr]\bigr)[S]$.
  Set $d \colonequals \deg(g)$;
  this is also the degree of $\theta(g)$ because the degree of $S^p - S - \tfrac{T}{\varpi^\flat}$ does not decrease under $\theta$.
  Since the roots of $S^p - S - \tfrac{T}{\varpi^\flat}$ in $\bigl(\C^\flat_p \llbracket T \rrbracket\bigr)[S]$ are given by $\alpha + \sum^{\infty}_{\nu = 0} \bigl(\tfrac{T}{\varpi^\flat}\bigr)^{p^\nu}$, where $\alpha$ ranges over all elements of $\F_p$, the $(d-1)$-st coefficient of $g$ is $-\sum^d_{\mu=1} \alpha_\mu - d \cdot \sum^\infty_{\nu = 0} \bigl(\tfrac{T}{\varpi^\flat}\bigr)^{p^\nu}$ for some (pairwise different) $\alpha_1,\dotsc,\alpha_d \in \F_p$.
  As the image of $\theta$ is contained in $\bigl(V \llbracket T \rrbracket\bigl[\tfrac{1}{\varpi^\flat}\bigr]\bigr)[S]$, where the coefficients in all appearing power series have bounded denominators, $d \equiv 0 \mod p$ and therefore $\deg(g) = 0$ or $\deg(h) = 0$.
  However, $S^p - S - \tfrac{T}{\varpi^\flat}$ is not divisible by any non-unit of $V\{T\}\bigl[\tfrac{1}{\varpi^\flat}\bigr]$, so either $g$ or $h$ must be a unit.
  In other words, $S^p - S - \tfrac{T}{\varpi^\flat}$ is irreducible and the torsor $V\{T\}\bigl[\tfrac{1}{\varpi^\flat}\bigr] \to \bigl(V\{T\}\bigl[\tfrac{1}{\varpi^\flat}\bigr]\bigr)[S]/\bigl(S^p - S - \tfrac{T}{\varpi^\flat}\bigr)$ is non-trivial.
\end{exmp}

\subsection{Base changing perfectoid covers along finite morphisms}

In this subsection, we record some results about certain finite covers of perfectoid spaces which will be crucial in the proof of Theorem~\ref{thm:perfectoid-Artin} and might be of independent interest.
We refer to \cite[\S~1.4]{MR1734903} for the definition and basic properties of finite morphisms in the generality of locally noetherian adic spaces.
If there is a finite morphism $Y_0 \to X_0$ in Situation~\ref{sit:inverse-system}, we do not know whether the fiber product $Y_0 \times_{X_0} X$ exists as an adic space:
it is unclear whether its structure presheaf is a sheaf.
Therefore, here it is advantageous to work in the category of diamonds in which this fiber product always exists and is in fact represented by another perfectoid space (Proposition~\ref{prop:perfectoidization}).

The next result, which holds more generally for any noetherian analytic affinoid adic space $X_0$ over $\Z_p$, is a direct consequence of Bhatt--Scholze's construction of perfectoidizations of mixed characteristic rings \cite{bhatt2019prisms}.
\begin{prop}\label{prop:perfectoidization-affinoid}
  Let $X_0 = \Spa(A_0,A^\circ_0)$ be an affinoid rigid space over $C$.
  Let $X = \Spa(A,A^+) \to X_0$ be an affinoid perfectoid space over $X_0$ and $Y_0 = \Spa(B_0,B^\circ_0) \to X_0$ be a finite morphism.
  Then the diamond $Y^\lozenge_0 \times_{X^\lozenge_0} X^\lozenge$ is the diamondification of an affinoid perfectoid space $Y$.
\end{prop}
\begin{proof}
  Since the diamondification functor from \cite[Lem.~15.1]{scholze2017etale} commutes with limits, $Y^\lozenge_0 \times_{X^\lozenge_0} X^\lozenge \simeq \Spd(B,B^+)$, where $(B,B^+)$ is the finite $(A,A^+)$-algebra with $B = B_0 \otimes_{A_0} A$ and $B^+$ the integral closure of $A^+$ in $B$ (cf.\ \cite[\S~1.4]{MR1734903}).
  As $A^+$ is perfectoid (in the sense of \cite[Def.~3.5]{MR3905467}) and $A^+ \to B^+$ is integral, there exists a universal map $B^+ \to B^+_\perfd$ from $B^+$ to a perfectoid ring \cite[Thm.~1.16.(1)]{bhatt2019prisms}.
  Consequently, the space $Y \colonequals \Spa\bigl(\widetilde{B},\widetilde{B}^+\bigr)$, where $\widetilde{B} \colonequals B^+_\perfd\bigl[\frac{1}{\varpi}\bigr]$ and $\widetilde{B}^+$ is the integral closure of $B^+_\perfd$ in $\widetilde{B}$, is perfectoid and the morphisms $Y \to Y_0$ and $Y \to X$ coming from the morphisms of affinoid Tate rings $(B_0,B^\circ_0) \to \bigl(\widetilde{B},\widetilde{B}^+\bigr)$ and $(A,A^+) \to \bigl(\widetilde{B},\widetilde{B}^+\bigr)$ induce an isomorphism of diamonds $Y^\lozenge \xrightarrow{\sim} Y^\lozenge_0 \times_{X^\lozenge_0} X^\lozenge$.
\end{proof}
The affinoid statement of Proposition~\ref{prop:perfectoidization-affinoid} globalizes by a standard gluing argument.
\begin{prop}\label{prop:perfectoidization}
  Let $X_0$ be a rigid space over $C$.
  Let $X \to X_0$ be a perfectoid space over $X_0$ and $f_0 \colon Y_0 \to X_0$ be a finite morphism.
  Then the diamond $Y^\lozenge_0 \times_{X^\lozenge_0} X^\lozenge$ is the diamondification of a perfectoid space $Y$.
\end{prop}
\begin{proof}
  Choose an open cover $X = \bigcup_{\alpha \in J} \widetilde{U}_\alpha$ by affinoid perfectoid spaces $\widetilde{U}_\alpha$ for which the compositions $\widetilde{U}_\alpha \to X_0$ factor through affinoid open subspaces $U_\alpha \subseteq X_0$.
  By Proposition~\ref{prop:perfectoidization-affinoid}, there are perfectoid spaces $f_\alpha \colon \widetilde{V}_\alpha \to \widetilde{U}_\alpha$ together with isomorphisms $\xi_\alpha \colon \widetilde{V}^\lozenge_\alpha \xrightarrow{\sim} \bigl(f^{-1}_0(U_\alpha)\bigr)^\lozenge \times_{U^\lozenge_\alpha} \widetilde{U}^\lozenge_\alpha \simeq Y^\lozenge_0 \times_{X^\lozenge_0} \widetilde{U}^\lozenge_\alpha$;
  as we work over a fixed perfectoid ground field, this simply means that $\widetilde{V}^\flat_\alpha$ represents the latter sheaf.

  For all $\alpha, \beta \in J$, set $\widetilde{U}_{\alpha\beta} \colonequals \widetilde{U}_\alpha \cap \widetilde{U}_\beta \subseteq X$, and similarly for triple intersections $\widetilde{U}_{\alpha\beta\gamma}$.
  Let $\widetilde{V}_{\alpha\beta} \colonequals f^{-1}_\alpha\bigl(\widetilde{U}_{\alpha\beta}\bigr) \subseteq \widetilde{V}_\alpha$.
  Since $\widetilde{V}_{\alpha\beta} \simeq \widetilde{V}_\alpha \times_{\widetilde{U}_\alpha} \widetilde{U}_{\alpha\beta}$, the tilt $\widetilde{V}^\flat_{\alpha\beta}$ represents
  \[ \widetilde{V}^\lozenge_\alpha \times_{\widetilde{U}^\lozenge_\alpha} \widetilde{U}^\lozenge_{\alpha\beta} \xrightarrow[\xi_\alpha\restriction\widetilde{V}^\flat_{\alpha\beta}]{\sim} Y^\lozenge_0 \times_{X^\lozenge_0} \widetilde{U}^\lozenge_\alpha \times_{\widetilde{U}^\lozenge_\alpha} \widetilde{U}^\lozenge_{\alpha\beta} \simeq Y^\lozenge_0 \times_{X^\lozenge_0} \widetilde{U}^\lozenge_{\alpha\beta}. \]
  The corresponding strings of isomorphisms for the perfectoid spaces $\widetilde{V}^\flat_{\beta\alpha} \subseteq \widetilde{V}^\flat_\beta$, $\widetilde{V}^\flat_{\alpha\beta} \cap \widetilde{V}^\flat_{\alpha\gamma} \subseteq \widetilde{V}^\flat_\alpha$, and $\widetilde{V}^\flat_{\beta\alpha} \cap \widetilde{V}^\flat_{\beta\gamma} \subseteq \widetilde{V}^\flat_\beta$ yield the following commutative diagram:
  \[ \begin{tikzcd}
      \Mor\bigl({-},\widetilde{V}^\flat_{\alpha\beta} \cap \widetilde{V}^\flat_{\alpha\gamma}\bigr) \arrow[r,"\sim"] \arrow[d] & Y^\lozenge_0 \times_{X^\lozenge_0} \widetilde{U}^\lozenge_{\alpha\beta\gamma} \arrow[d] & \Mor\bigl({-},\widetilde{V}^\flat_{\beta\alpha} \cap \widetilde{V}^\flat_{\beta\gamma}\bigr) \arrow[l,"\sim"'] \arrow[d] \\
      \Mor\bigl({-},\widetilde{V}^\flat_{\alpha\beta}\bigr) \arrow[r,"\xi_\alpha\restriction\widetilde{V}^\flat_{\alpha\beta}","\sim"'] & Y^\lozenge_0 \times_{X^\lozenge_0} \widetilde{U}^\lozenge_{\alpha\beta} & \Mor\bigl({-},\widetilde{V}^\flat_{\beta\alpha}\bigr). \arrow[l,"\sim","\xi_\beta\restriction\widetilde{V}^\flat_{\beta\alpha}"']
  \end{tikzcd} \]
  By the full faithfulness of the Yoneda functor, the bottom lines are induced by unique isomorphisms $\varphi_{\alpha\beta} \colon \widetilde{V}^\flat_{\alpha\beta} \xrightarrow{\sim} \widetilde{V}^\flat_{\beta\alpha}$ such that $\varphi^{-1}_{\alpha\beta}\bigl(\widetilde{V}^\flat_{\beta\alpha} \cap \widetilde{V}^\flat_{\beta\gamma}\bigr) = \widetilde{V}^\flat_{\alpha\beta} \cap \widetilde{V}^\flat_{\alpha\gamma}$ and the cocycle condition
  \[ \restr{\varphi_{\beta\gamma}}{\widetilde{V}^\flat_{\beta\alpha} \cap \widetilde{V}^\flat_{\beta\gamma}} \circ \restr{\varphi_{\alpha\beta}}{\widetilde{V}^\flat_{\alpha\beta} \cap \widetilde{V}^\flat_{\alpha\gamma}} = \restr{\varphi_{\alpha\gamma}}{\widetilde{V}^\flat_{\alpha\beta} \cap \widetilde{V}^\flat_{\alpha\gamma}} \]
  is satisfied.
  Using these isomorphisms and the tilting equivalence, we can glue the various $\widetilde{V}_\alpha$ to a perfectoid space $Y$.

  Moreover, the $\xi_\alpha$ glue to a morphism $\xi \colon Y^\lozenge \to Y^\lozenge_0 \times_{X^\lozenge_0} X^\lozenge$ because $\xi_\beta \circ \varphi^\lozenge_{\alpha\beta} = \restr{\xi_\alpha}{\widetilde{V}^\flat_{\alpha\beta}}$ essentially by definition.
  We can check that $\xi$ is an isomorphism after base change along the open cover $Y^\lozenge_0 \times_{X^\lozenge_0} \widetilde{U}^\lozenge_\alpha$.
  In that case, this follows from the fact that the $\xi_\alpha$ are isomorphisms.
\end{proof}
\begin{exmp}\label{exmp:closed-subspace}
  When $Y_0 \to X_0$ is a Zariski-closed immersion with corresponding coherent ideal sheaf $\cJ_0 \subset \cO_{X_0}$, the perfectoid space $Y$ is given in the affinoid setting of Proposition~\ref{prop:perfectoidization-affinoid} by the Zariski-closed subspace corresponding to the ideal $\cJ_0(\Spa(A_0,A^\circ_0))\cdot A \subseteq A$ with its perfectoid structure from \cite[Lem.~2.2.2]{MR3418533}; cf.\ also \cite[Rmk.~7.5]{bhatt2019prisms}.
  In the general setting of Proposition~\ref{prop:perfectoidization}, these spaces glue again to $Y \subset X$ by universality.
\end{exmp}
Proposition~\ref{prop:perfectoidization} shows in particular that Situation~\ref{sit:inverse-system} is preserved under base change along any finite morphism $Y_0 \to X_0$.
The isomorphism $Y^\lozenge \simeq Y^\lozenge_0 \times_{X^\lozenge_0} X^\lozenge$ yields, via projection to the two factors and the definition of the diamonds $Y^\lozenge_0$ and $X^\lozenge$, morphisms of adic spaces $Y \to Y_0$ and $Y \to X$.
Setting $Y_i \colonequals Y_0 \times_{X_0} X_i$ for all $i \in I$, we obtain an inverse system of maps $Y \to Y_i$.
\begin{lem}\label{lem:system-etale-sheaves}
  For all \'etale sheaves $F_0$ on $Y_0$ with pullback $F_i$ to $Y_i$ and $F$ to $Y$, the natural map $\colim_i \Hh^n(Y_i,F_i) \to \Hh^n(Y,F)$ is an isomorphism.
\end{lem}
\begin{proof}
  By \cite[Lem.~15.6]{scholze2017etale}, the diamonds $Y^\lozenge_i$ associated with $Y_i$ are spatial and the diamondification functor induces a natural equivalence of sites $Y^\lozenge_{i,\et} \xrightarrow{\sim} Y_{i,\et}$.
  Denoting the pullbacks of the $F_i$ and $F$ under these equivalences by $F^\lozenge_i$ and $F^\lozenge$, respectively, it therefore suffices to show that the natural map
  \[ \colim_i \Hh^n_{\et}\bigl(Y^\lozenge_i,F^\lozenge_i\bigr) \to \Hh^n_{\et}\bigl(Y^\lozenge,F^\lozenge\bigr) \]
  is an isomorphism.
  But
  \[ \lim_i Y^\lozenge_i = Y^\lozenge_0 \times_{X^\lozenge_0} \lim_i X^\lozenge_i \simeq Y^\lozenge_0 \times_{X^\lozenge_0} X^\lozenge, \]
  so the statement follows from \cite[Prop.~14.9]{scholze2017etale} and Proposition~\ref{prop:perfectoidization}.
\end{proof}
\begin{exmp}\label{exmp:system-structure-sheaves} 
  By \cite[\S~2.3]{caraiani2018shimura}, any locally spatial diamond $\Delta$ over $\Spd(C,\cO_C)$ can be equipped with an integral structure sheaf $\cO^+_\Delta$ of $\cO_C$-modules on its quasi-pro-\'etale site $\Delta_{\qproet}$, defined in \cite[Def.~14.1]{scholze2017etale}.
  Moreover, the quotient sheaf $\cO^+_\Delta/\varpi$ is \'etale, that is, in the image of the fully faithful pullback functor $\Delta^\sim_{\et} \to \Delta^\sim_{\qproet}$.
  When $\Delta = Z^\lozenge$ is the diamondification of a rigid space $Z$, the pullback functor $Z^\sim_{\et} \to Z^{\lozenge,\sim}_{\et}$ (obtained from the equivalence of site $Z^\lozenge_{\et} \xrightarrow{\sim} Z_{\et}$) identifies the sheaf of almost $\cO_C$-modules $\bigl(\cO^+_Z/\varpi\bigr)^a$ with $\bigl(\cO^+_{Z^\lozenge}/\varpi\bigr)^a$ \cite[Lem.~2.3.3]{caraiani2018shimura}.

  This discussion applies to the setting of Lemma~\ref{lem:system-etale-sheaves} as follows:
  When $F_0 = \bigl(\cO^+_{Y_0}/\varpi\bigr)^a$, chasing through the diagram
  \[ \begin{tikzcd}
      Y^\lozenge_{i,\qproet} \arrow[r] \arrow[d] & Y^\lozenge_{i,\et} \arrow[r,"\sim"] \arrow[d] & Y_{i,\et} \arrow[d,"\rho_i"] \\
      Y^\lozenge_{0,\qproet} \arrow[r] & Y^\lozenge_{0,\et} \arrow[r,"\sim"] & Y_{0,\et}
  \end{tikzcd} \]
  shows that the pullback $F_i = \rho^*_i F_0$ is given by $\bigl(\cO^+_{Y_i}/\varpi\bigr)^a$ because the diamondification $\rho^\lozenge_i \colon Y^\lozenge_i \to Y^\lozenge_0$ of the finite morphism $Y_i \to Y_0$ is quasi-pro-\'etale \cite[Prop.~2.3.4]{caraiani2018shimura};
  virtually the same argument applies to the pullback to $Y_{\et}$.
  Therefore, the natural map $\colim_i \Hh^n\bigl(Y_i,\cO^+_{Y_i}/\varpi\bigr) \to \Hh^n\bigl(Y,\cO^+_Y/\varpi\bigr)$ is an almost isomorphism by Lemma~\ref{lem:system-etale-sheaves}.
\end{exmp}

\subsection{Proof of Theorem~\ref{thm:perfectoid-Artin}}

We begin with the following key assertion, whose proof is inspired by \cite[\S~4.2]{MR3418533}.
\begin{lem}\label{lem:perfectoid-Artin-key}
  In Situation~\ref{sit:inverse-system}, let $Y_0 \to X_0$ be a finite morphism and $Y_i \colonequals Y_0 \times_{X_0} X_i$ for all $i \in I$.
  Assume that $X_0$ is proper.
  Then for all $n > d$
  \[ \colim_i \Hh^n_{\et}(Y_i,\F_p) = 0. \]
\end{lem}
\begin{proof}
  It suffices to show that the free $\cO_C/p$-module
  \[ \colim_i \Hh^n_{\et}(Y_i,\F_p) \otimes_{\F_p} \cO_C/p \]
  is almost zero for all $n>d$.
  By Scholze's primitive comparison theorem \cite[Thm.~3.17]{MR3204346} (here properness of the $X_i$ is crucial!) and Example~\ref{exmp:system-structure-sheaves}, there is an almost isomorphism
  \[ \colim_i \Hh^n_{\et}(Y_i,\F_p) \otimes_{\F_p} \cO_C/p \xrightarrow{\sim} \colim_i \Hh^n_{\et}(Y_i,\cO^+_{Y_i}/p) \xrightarrow{\sim} \Hh^n_{\et}(Y,\cO^+_Y/p), \]
  where $Y$ is the perfectoid space from Proposition~\ref{prop:perfectoidization}.
  The \v{C}ech-to-derived functor spectral sequence for an affinoid perfectoid cover of $Y$ and \cite[Prop.~6.14, Prop.~7.13]{MR3090258} or \cite[Prop.~8.3.2.(c)]{MR3379653} show that there is an almost isomorphism of cohomology groups $\Hh^n_{\et}(Y,\cO^+_Y/p)^a \simeq \Hh^n(Y,\cO^+_Y/p)^a$ for the \'etale and analytic topology.
  However, since $Y^\lozenge \simeq Y^\lozenge_0 \times_{X^\lozenge_0} X^\lozenge \simeq \lim_i Y^\lozenge_i$ and thus $\abs{Y} \simeq \lim_i \abs{Y_i}$ \cite[Lem.~11.22, Lem.~15.6]{scholze2017etale} and since the cohomological dimension of the $\abs{Y_i}$ is at most $\dim Y_i = d$ \cite[Cor.~4.6]{MR1179103}, the statement follows; cf. the proof of \cite[Cor.~4.2.2]{MR3418533}.
\end{proof}
\begin{rem}
  In the proof of Lemma~\ref{lem:perfectoid-Artin-key}, we could have proceeded along the lines of \cite{caraiani2018shimura} and used Proposition~\ref{prop:perfectoidization-affinoid} directly instead of referring to the global statement of Proposition~\ref{prop:perfectoidization}.
  Namely, we can simply define the spatial diamond $Y^\lozenge \colonequals \lim_i Y^\lozenge_i$;
  the natural map
  \[ \colim_i \Hh^n_{\et}(Y^\lozenge_i,\cO^+_{Y^\lozenge_i}/p) \to \Hh^n_{\et}(Y^\lozenge,\cO^+_{Y^\lozenge}/p) \]
  is still an isomorphism by \cite[Prop.~14.9]{scholze2017etale}.
  The morphism $\pi \colon Y^\lozenge_{\et} \to \abs{Y^\lozenge_0} = \abs{Y_0}$ from the \'etale site of $Y^\lozenge$ to the analytic site of $Y_0$ induces an isomorphism
  \[ \R\Gamma_{\et}(Y^\lozenge,\cO^+_{Y^\lozenge}/p) \simeq \R\Gamma\bigl(\abs{Y_0},\R\pi_*(\cO^+_{Y^\lozenge}/p)\bigr). \]
  Since the cohomological dimension of $\abs{Y_0}$ is at most $\dim Y_0 = d$ \cite[Cor.~4.6]{MR1179103}, we are left to prove that $\Rr^i\pi_*(\cO^+_{Y^\lozenge}/p)^a = 0$ for all $i>0$.
  This statement can be checked locally on $\abs{Y_0}$ \cite[Cor.~16.10]{scholze2017etale}, so we may assume that we are in the affinoid situation of Proposition~\ref{prop:perfectoidization-affinoid} and conclude with \cite[Prop.~7.13]{MR3090258} or \cite[Prop.~8.3.2.(c)]{MR3379653}.
\end{rem}
At last, we can finish the proof of Theorem~\ref{thm:perfectoid-Artin} by reducing the assertion to Lemma~\ref{lem:perfectoid-Artin-key} via a standard d\'evissage argument.
\begin{proof}[{Proof of Theorem~\ref{thm:perfectoid-Artin}}]
  It suffices to prove the following statement for all $d \ge 0$ and all $n > d$ via ascending induction on $d$ and descending induction on $n$:
  \begin{equation}\label{eqn:induction-claim}
    \colim_k \Hh^n_{\et}(X_k,F_k) = 0.
    \tag{$V_{d,n}$}
  \end{equation}
  The base case $n > 2d$ is \cite[Cor.~2.8.3]{MR1734903}.
  For the inductive step, we fix $d > 0$ and $n > d$ and assume $(V_{d',n'})$ for all $d'$, $n'$ such that either $d' < d$ and $n' > d'$, or $d' = d$ and $n' > n$.

  First, by the topological invariance of the \'etale site and Proposition~\ref{prop:perfectoidization}, we may, after possibly replacing $X_k$ with $X_{0,\red} \times_{X_0} X_k$, assume that $X_0$ is reduced.
  Choose a dense, Zariski-open subspace $j_0 \colon U_0 \hookrightarrow X_0$ for which $j^*_0F_0$ is locally constant.
  By shrinking $U_0$ if necessary, we can ensure that $U_0$ is normal; cf.\ \cite[Thm.~2.1.2]{MR1697371}.
  Let $i_0 \colon Z_0 \hookrightarrow X_0$ be the inclusion of the complement of $U_0$ with its reduced closed subspace structure.
  For all $k \in I$, define $U_k \colonequals U_0 \times_{X_0} X_k$, $Z_k \colonequals Z_0 \times_{X_0} X_k$, and let $j_k \colon U_k \hookrightarrow X_k$ and $i_k \colon Z_k \hookrightarrow X_k$ be the projections to the second factor.

  By Proposition~\ref{prop:perfectoidization}, there is a perfectoid space $Z$ representing $\lim_k Z^\lozenge_k$.
  Thus, by the induction hypothesis, $\colim_k \Hh^{n'}_{\et}(X_k,i_{k,*}i^*_kF_k) = \colim_k \Hh^{n'}_{\et}(Z_k,\pi^*_{k0}i^*_0F_0) = 0$ for $n' > d-1$.
  Taking cohomology of the direct system of short exact sequences
  \[ 0 \to j_{k,!}j^*_kF_k \to F_k \to i_{k,*}i^*_kF_k \to 0, \]
  we see that the resulting map
  \[ \colim_k \Hh^n_{\et}(X_k,j_{k,!}j^*_kF_k) \to \colim_k \Hh^n_{\et}(X_k,F_k) \]
  is an isomorphism.
  Since $j_{k,!}j^*_kF_k \simeq \pi^*_{k0}j_{0,!}j^*_0F_0$ by proper base change, we may assume that $F_0 = j_{0,!} L$ for some local system $L$ on $U_0$.
  
  In this case, we can choose a finite \'etale cover $f_0 \colon V_0 \to U_0$ for which $f^*_0L \simeq \F^{\oplus r}_p$.
  Let $\nu_0 \colon \widetilde{X}_0 \to X_0$ be the normalization of $X_0$;
  then $\nu_0$ is an isomorphism over $U_0$ (see e.g.\ \cite[\S~2]{MR1697371} for basic facts about normalizations of rigid spaces).
  By \cite[Thm.~1.6]{hansen-vanishing}, $f_0$ can be extended to a finite cover $\bar{f}_0 \colon Y_0 \to \widetilde{X}_0$.
  Denote by $\bar{\jmath}_0 \colon V_0 \hookrightarrow Y_0$ the induced open immersion.
  The trace map yields a surjective morphism
  \begin{equation}\label{eqn:trace-surjection}
    \nu_{0,*}\bar{f}_{0,*}\bar{\jmath}_{0,!} \F^{\oplus r}_p \simeq j_{0,!}f_{0,*} \F^{\oplus r}_p \simeq j_{0,!}f_{0,*}f^*_0 L \twoheadrightarrow j_{0,!} L.
  \end{equation}
  
  Let $K$ denote the kernel of this map.
  For all $k \in I$, set $\widetilde{X}_k \colonequals \widetilde{X}_0 \times_{X_0} X_k$, $V_k \colonequals V_0 \times_{X_0} X_k$, and $Y_k \colonequals Y_0 \times_{X_0} X_k$.
  The projections $\nu_k \colon \widetilde{X}_k \to X_k$ are still isomorphisms over $U_k$.
  We obtain finite \'etale covers $f_k \colon V_k \to U_k$ extending to finite covers $\bar{f}_k \colon Y_k \to \widetilde{X}_k$ and open immersions $\bar{\jmath}_k \colon V_k \hookrightarrow Y_k$ as base changes from the respective morphisms over $X_0$.

  Since $\nu_0$, $f_0$, and $\bar{f}_0$ are finite, proper base change shows that pulling back (\ref{eqn:trace-surjection}) along $\pi_{k0}$ yields a direct system of short exact sequences
  \[ 0 \to \pi^*_{k0}K \to \nu_{k,*}\bar{f}_{k,*}\bar{\jmath}_{k,!} \F^{\oplus r}_p \to \pi^*_{k0}j_{0,!}L \to 0, \]
  From the induction hypothesis $(V_{d,n+1})$, we know that $\colim_k \Hh^{n+1}_{\et}(X_k,\pi^*_{k0}K) = 0$.
  It remains to show that $\colim_k \Hh^n_{\et}(X_k,\nu_{k,*}\bar{f}_{k,*}\bar{\jmath}_{k,!} \F^{\oplus r}_p) = \colim_k \Hh^n_{\et}(Y_k,\bar{\jmath}_{k,!} \F^{\oplus r}_p) = 0$ as well.
  
  Let $A_k \colonequals Z_k \times_{X_k} Y_k$ and $\bar{\imath}_k \colon A_k \hookrightarrow Y_k$ be the projection to the second factor.
  Consider the short exact sequences
  \[ 0 \to \bar{\jmath}_{k,!}\F^{\oplus r}_p \to \F^{\oplus r}_p \to \bar{\imath}_{k,*}\F^{\oplus r}_p \to 0 \]
  on $Y_k$.
  Lemma~\ref{lem:perfectoid-Artin-key} applied to the systems of finite morphisms $Y_k \to X_k$ and $A_k \to Z_k$ shows that $\colim_k \Hh^{n-1}\bigl(Y_k,\bar{\imath}_{k,*}\F^{\oplus r}_p\bigr) = \bigl(\colim_k \Hh^{n-1}(A_k,\F_p)\bigr)^{\oplus r} = 0$ and $\colim_k \Hh^n_{\et}\bigl(Y_k,\F^{\oplus r}_p\bigr) = 0$, respectively, whence the claim.
\end{proof}

\section{Vanishing from ramification of the boundary divisor}\label{sect:boundary-vanishing}

Throughout this section, we work over a fixed algebraically closed field $k$.
\begin{defn}\label{defn:nc-divisor}
  Let $\cY$ be a smooth Deligne--Mumford stack over $k$.
  A \emph{normal crossings divisor} on $\cY$ is an effective Cartier divisor
  $\cD \subset \cY$ such that for any $y \in \cY(k)$, there is an integer $0 \le \gamma \le d$ and an isomorphism
  \[ \cO^\wedge_{\cY,y} \simeq k\llbracket t_1,\dotsc,t_d\rrbracket \]
  under which the restriction of $\cD$ to $\Spec \cO^\wedge_{\cY,y}$ is identified with $\{ t_1\dotsm t_\gamma = 0 \}$.
\end{defn}
Note that if $\cY$ is a scheme, Definition~\ref{defn:nc-divisor} agrees with the usual one by \cite[\href{https://stacks.math.columbia.edu/tag/0CBS}{Lem.~0CBS}]{stacks-project} and Artin approximation.
We will use the following notion from \cite[Def.~4.2]{MR3783421}.
\begin{defn}\label{defn:m-ramified}
  Let $\pi \colon \cY' \to \cY$ be a proper and quasi-finite morphism of smooth Deligne--Mumford stacks over $k$.
  Let $\cD \subset \cY$ be a normal crossings divisor.
  Then $\pi$ is called \emph{$m$-ramified over $\cD$} if for all $y \in \cY'(k)$ the induced map $\pi^* \colon \cO^\wedge_{\cY,\pi(y)} \to \cO^\wedge_{\cY',y}$ can be identified with a morphism
  \[ k\llbracket t_1,\dotsc,t_d\rrbracket \to k\llbracket t_1,\dotsc,t_d\rrbracket \]
  such that the restriction of $\cD$ to $\Spec \cO^\wedge_{\cY,\pi(y)}$ is given by $\{ t_1 \dotsm t_\gamma = 0 \}$ and there exist units $\nu_i \in k\llbracket t_1,\dotsc,t_d\rrbracket$ with $\pi^*(t_i) = \nu_i t^m_i$ for all $1 \le i \le \gamma$.
\end{defn}
\begin{exmp}\label{exmp:coordinate-hyperplanes}
  If $\cY' = \cY = \PP^N_k$, the mock Frobenius
  \[ \pi \colon \PP^N_k \to \PP^N_k, \quad [x_0:\dotsb:x_N] \mapsto [x^p_0:\dotsb:x^p_N] \]
  from Example~\ref{exmp:proj-perfectoid} is $p$-ramified over the divisor of coordinate hyperplanes $\cD \colonequals \{ x_0 \dotsm x_N = 0 \}$.
\end{exmp}
From now on assume $\charac k \neq p$.
The importance of Definition~\ref{defn:m-ramified} lies in the following calculation.
\begin{lem}\label{lem:m-ramified-vanishing}
  Let $Y \colonequals \Spec k\llbracket t_1,\dotsc,t_d\rrbracket$ and $D$ be the divisor $\{t_1 \dotsm t_\gamma = 0\}$ for some $0 \le \gamma \le d$.
  Set $U \colonequals Y \smallsetminus D$.
  Let $\pi \colon Y \to Y$ be a morphism such that for all $1 \le r \le \gamma$ there exist units $\nu_r \in k\llbracket t_1,\dotsc,t_d\rrbracket$ with $\pi^*(t_r) = \nu_r t^p_r$.
  Then the induced map
  \[ \pi^*_U \colon \Hh^i(U,\F_p) \to \Hh^i(U,\F_p) \]
  is $0$ for all $i>0$.
\end{lem}
\begin{proof}
  For all $1 \le r \le \gamma$, let $U_r \colonequals Y \smallsetminus \{t_r = 0\}$.
  By cohomological purity (and the assumption on $\charac k$), $\Hh^1(U_r,\F_p) \simeq \F_p$ and the natural map
  \[ \bigwedge^i \bigl(\oplus^\gamma_{r=1} \Hh^1(U_r,\F_p)\bigr) \to \Hh^i (U,\F_p) \]
  given by the cup product is an isomorphism;
  see e.g.\ \cite[Thm.~XIX.1.2]{SGA4} or \cite[Cor.~XVI.3.1.4]{MR3309086} for the most general version.
  Thus, we may assume $\gamma = 1$ and $i = 1$.

  Since $z^p - t_1$ is irreducible in $\bigl(k\llbracket t_1,\dotsc,t_d\rrbracket_{t_1}\bigr)[z]$, the group $\Hh^1(U,\F_p)$ is (under the non-canonical isomorphism $\F_p \simeq \mu_p$) generated by the class of the $\mu_p$-torsor
  \[ \Spec \tfrac{k\llbracket t_1,\dotsc,t_d\rrbracket_{t_1}[z]}{(z^p - t_1)} \to \Spec k\llbracket t_1,\dotsc,t_d\rrbracket_{t_1}. \]
  Its image under $\pi^*_U$ is the class of $\mu_p$-torsor 
  \[ \Spec \tfrac{k\llbracket t_1,\dotsc,t_d\rrbracket_{t_1}[z]}{(z^p - \nu_1t^p_1)} \to \Spec k\llbracket t_1,\dotsc,t_d\rrbracket_{t_1}. \]
  However, this torsor is trivial because the unit $\nu_1$ has a $p$-th root $\lambda_1 \in k\llbracket t_1,\dotsc,t_d \rrbracket$ by Hensel's lemma and thus
  \[ z^p - \nu_1t^p_1 = \prod_{\zeta \in \mu_p(k)} (z - \zeta \lambda_1 t_1). \qedhere \]
\end{proof}
\begin{thm}\label{thm:tower-cohomology}
  Let $\cY_n$, $n \in \Z_{\ge 0}$, be a projective system of smooth, tame Deligne--Mumford stacks of finite type over $k$ with proper and quasi-finite transition maps $\pi_{mn} \colon \cY_m \to \cY_n$ for all $m \ge n$.
  Let $\Lambda_0$ be an \'etale $\F_p$-local system on $\cY_0$ and $\Lambda_n \colonequals \pi^*_{n0} \Lambda_0$.
  Assume there exist normal crossings divisors $\cD_n \subset \cY_n$ such that $\cD_m = (\pi^*_{mn} \cD_n)_{\red}$ and $\pi_{n+1,n}$ is $p$-ramified over $\cD_n$ for all $n$.
  Set $\cU_n \colonequals \cY_n \smallsetminus \cD_n$.
  Then the natural map
  \[ \colim_n \Hh^i(\cY_n,\Lambda_n) \to \colim_n \Hh^i(\cU_n,\Lambda_n) \]
  is an isomorphism for all $i \ge 0$.
\end{thm}
Throughout the proof, we denote by $j_n \colon \cU_n \hookrightarrow \cY_n$ and $i_n \colon \cD_n \hookrightarrow \cY_n$ the canonical inclusions.
To simplify notation, we moreover set $\pi_n \colonequals \pi_{n0}$.
The strategy of the proof is as follows.
By the Gysin sequence, the assertion is implied by $\hocolim_n \R\pi_{n,*} i_{n,*} \R i^!_n \Lambda_n \simeq 0$, which amounts to the vanishing of $\colim_n \cH^\ell\bigl(\R\pi_{n,*}i_{n,*}\R i^!_n \Lambda_n\bigr)_{y_0}$ for all $\ell$ and all $y_0 \in \cY(k)$.
In fact, we show that any composition of $\ell$ transition maps in the latter directed systems is $0$, which can be checked on completions by a standard reduction.

Thus, we first treat the complete, local picture.
Let $y_0 \in \cY_0(k)$.
Choose a system of points $y_m \in \cY_m(k)$ with $\pi_{m+1,m}(y_{m+1}) = y_m$ for all $m \ge 0$.
Set $V_m \colonequals \Spec \cO^\wedge_{\cY_m,y_m}$.
Let $Z_m \subset V_m$ be the restriction of $\cD_m$ to $V_m$ and $U_m \colonequals V_m \smallsetminus Z_m \subseteq V_m$ be its complement.

Let $H_m \colonequals \Aut(y_m)$ be the automorphism group scheme of $y_m$.
It acts naturally on $V_m$.
Since $\cD_m \subset \cY_m$ is a substack, $U_m$ and $Z_m$ are invariant under this action.
We obtain diagrams
\[ \begin{tikzcd}
    \left[U_m/H_m\right] \arrow[r,hook,"\hat{\jmath}_m"] \arrow[d] & \left[V_m/H_m\right] \arrow[d,"\hat{\pi}_m"] & \left[Z_m/H_m\right] \arrow[l,hook',"\hat{\imath}_m"'] \arrow[d] \\
    \left[U_0/H_0\right] \arrow[r,hook,"\hat{\jmath}_0"] & \left[V_0/H_0\right] & \left[Z_0/H_0\right] \arrow[l,hook',"\hat{\imath}_0"'],
\end{tikzcd} \]
where $\hat{\jmath}_m$ is an open and $\hat{\imath}_m$ a closed immersion.
Denote the pullback of $\Lambda_m$ to $\left[V_m/H_m\right]$ by $\widehat{\Lambda}_m$;
we have $\widehat{\Lambda}_m = \hat{\pi}^*_m \widehat{\Lambda}_0$ for all $m \ge 0$.
\begin{lem}\label{lem:local-ramification}
  For any $\ell,n \ge 0$, the pullback morphism
  \[ \cH^\ell\bigl((\R\hat{\pi}_{n,*}\hat{\imath}_{n,*}\R \hat{\imath}^!_n \widehat{\Lambda}_n)_{y_0}\bigr) \to \cH^\ell\bigl((\R\hat{\pi}_{n+\ell,*}\hat{\imath}_{n+\ell,*}\R \hat{\imath}^!_{n+\ell} \widehat{\Lambda}_{n+\ell})_{y_0}\bigr) \]
  of $\F_p$-vector spaces is $0$.
\end{lem}
Here, the stalk at $y_0$ is understood to be the derived pullback under $y_0 \colon \Spec k \to \cY_0$.
\begin{proof}
  For any $m \ge 0$, the exceptional inverse image fits into the distinguished triangle
  \begin{equation}\label{eqn:local-cohom-triangle}
    (\R\hat{\pi}_{m,*}\widehat{\Lambda}_m)_{y_0} \to (\R\hat{\pi}_{m,*}\R \hat{\jmath}_{m,*}\hat{\jmath}^*_m \widehat{\Lambda}_m)_{y_0} \to \bigl(\R\hat{\pi}_{m,*}\hat{\imath}_{m,*}\R \hat{\imath}^!_m \widehat{\Lambda}_m\bigr)_{y_0}[1] \to.
  \end{equation}
  Set $K_m \colonequals \ker(H_m \to H_0)$.
  Via base change (cf.\ e.g.\ \cite[Thm.~A.0.2]{MR2007376}) across the diagram of fiber squares
  \[ \begin{tikzcd}
      & \left[U_m/K_m\right] \arrow[d,hook,"\hat{\jmath}'"] \arrow[r,"\alpha'"] & \left[U_m/H_m\right] \arrow[d,hook,"\hat{\jmath}_m"] \\
      \left[\Spec k/K_m\right] \arrow[d,"\hat{\rho}"] \arrow[r,"\tilde{y}_m"] & \left[V_m/K_m\right] \arrow[d,"\hat{\pi}'"] \arrow[r,"\alpha_m"] & \left[V_m/H_m\right] \arrow[d,"\hat{\pi}_m"] \\
      \Spec k \arrow[r,"\tilde{y}_0"] & V_0 \arrow[r,"\alpha_0"] & \left[V_0/H_0\right]
  \end{tikzcd} \]
  (in which $\tilde{y}_0$ is the lift of $y_0$ and $\alpha_0$ is the canonical \'etale atlas), (\ref{eqn:local-cohom-triangle}) is identified with
  \[ \R\hat{\rho}_*\tilde{y}^*_m\hat{\pi}^{\prime,*}\alpha^*_0\widehat{\Lambda}_0 \to \R\hat{\rho}_*\tilde{y}^*_m\R \hat{\jmath}'_*\hat{\jmath}^{\prime,*}\hat{\pi}^{\prime,*}\alpha^*_0\widehat{\Lambda}_0 \to \bigl(\R\hat{\pi}_{m,*}\hat{\imath}_{m,*}\R \hat{\imath}^!_m \widehat{\Lambda}_m\bigr)_{y_0}[1] \to. \]

  Since $V_0$ is strictly henselian, $\alpha^*_0\widehat{\Lambda}_0 \simeq \F^{\oplus r}_p$, where $r$ is the rank of $\Lambda_0$.
  Thus, $\tilde{y}^*_m\hat{\pi}^{\prime,*}\alpha^*_0\widehat{\Lambda}_0 \simeq \F^{\oplus r}_p$ with the trivial $K_m$-action and $\tilde{y}^*_m\R \hat{\jmath}'_*\hat{\jmath}^{\prime,*}\hat{\pi}^{\prime,*}\alpha^*_0\widehat{\Lambda}_0 \simeq \R\Gamma(U_m,\F^{\oplus r}_p)$ with $K_m$-action induced by the $K_m$-action on $U_m$.
  Moreover, $\tilde{y}^*_m\hat{\pi}^{\prime,*}\alpha^*_0\widehat{\Lambda}_0 \to \tilde{y}^*_m\R \hat{\jmath}'_*\hat{\jmath}^{\prime,*}\hat{\pi}^{\prime,*}\alpha^*_0\widehat{\Lambda}_0$ is the natural morphism $\tau^{\leq 0}\bigl(\R\Gamma(U_m,\F^{\oplus r}_p)\bigr) \to \R\Gamma(U_m,\F^{\oplus r}_p)$, where $\tau$ denotes the truncation with respect to the natural t-structure on the derived category of $\F_p[K_m]$-modules \cite[Prop.~1.3.3]{MR751966}.
  This proves
  \[ \bigl(\R\hat{\pi}_{m,*}\hat{\imath}_{m,*}\R \hat{\imath}^!_m\widehat{\Lambda}_m\bigr)_{y_0} \simeq \R\hat{\rho}_*\bigl(\tau^{>0}\bigl(\R\Gamma(U_m,\F^{\oplus r}_p)\bigr)\bigr)[-1] \simeq \R\Gamma\bigl(K_m,\tau^{>0}\bigl(\R\Gamma(U_m,\F^{\oplus r}_p)\bigr)\bigr)[-1] \]
  and hence $\cH^\ell\bigl((\R\hat{\pi}_{m,*}\hat{\imath}_{m,*}\R \hat{\imath}^!_m\widehat{\Lambda}_m)_{y_0}\bigr) \simeq \Hh^{\ell-1}\bigl(K_m,\tau^{>0}\bigl(\R\Gamma(U_m,\F^{\oplus r}_p)\bigr)\bigr)$.

  It remains to show that the pullback morphism
  \[ \hat{\pi}^*_{n+\ell,n} \colon \Hh^{\ell-1}\bigl(K_n,\tau^{>0}\bigl(\R\Gamma(U_n,\F^{\oplus r}_p)\bigr)\bigr) \to \Hh^{\ell-1}\bigl(K_{n+\ell},\tau^{>0}\bigl(\R\Gamma(U_{n+\ell},\F^{\oplus r}_p)\bigr)\bigr) \]
  is $0$ for all $\ell$.
  For $n \le m \le n+\ell-1$, $\hat{\pi}^*_{m+1,m}$ is induced by the natural morphism of hypercohomology spectral sequences
  \[ \begin{tikzcd}
      \Hh^p\bigl(K_m,\cH^q\bigl(\tau^{>0}(\R\Gamma(U_m,\F^{\oplus r}_p))\bigr)\bigr) \arrow[r,phantom,"\Longrightarrow"] \arrow[d,"\cH^q(\hat{\pi}^*_{m+1,m})"] & \Hh^{\ell-1}\bigl(K_m,\tau^{>0}(\R\Gamma(U_m,\F^{\oplus r}_p))\bigr) \arrow[d,"\hat{\pi}^*_{m+1,m}"] \\
      \Hh^p\bigl(K_{m+1},\cH^q\bigl(\tau^{>0}(\R\Gamma(U_{m+1},\F^{\oplus r}_p))\bigr)\bigr) \arrow[r,phantom,"\Longrightarrow"] & \Hh^{\ell-1}(K_{m+1},\tau^{>0}(\R\Gamma(U_{m+1},\F^{\oplus r}_p))\bigr).
  \end{tikzcd} \]
  This follows from the functoriality properties of the Cartan--Eilenberg resolution, which is used in the construction of the hypercohomology spectral sequence, and of the injective resolutions from which the pullback maps are computed.

  The morphisms
  \[ \cH^q\bigl(\R\Gamma(U_m,\F^{\oplus r}_p)\bigr) \to \cH^q\bigl(\R\Gamma(U_{m+1},\F^{\oplus r}_p)\bigr) \]
  of $\F_p[K_{m+1}]$-modules are $0$ for all $q > 0$:
  this can be checked on the \'etale cover $\Spec k \to \left[\Spec k/K_{m+1}\right]$, where it follows from Definition~\ref{defn:m-ramified}, Lemma~\ref{lem:m-ramified-vanishing}, and the additivity of the cohomology functors.
  Therefore, the morphisms between the second pages are $0$, and so are the morphisms between the graded rings associated to the induced $\ell$-step filtrations of the abutments.
  Consequently, $\hat{\pi}^*_{n+\ell,n} = \hat{\pi}^*_{n+\ell,n+\ell-1} \circ \dotsb \circ \hat{\pi}^*_{n+1,n} = 0$.
\end{proof}
We return to the setting of Theorem~\ref{thm:tower-cohomology}.
\begin{lem}\label{lem:global-ramification}
  For all $y_0 \in \cY_0(k)$ and all $\ell,n \ge 0$, the pullback morphism
  \[ \cH^\ell\bigl(\R\pi_{n,*}i_{n,*}\R i^!_n\Lambda_n\bigr)_{y_0} \to \cH^\ell\bigl(\R\pi_{n+\ell,*}i_{n+\ell,*}\R i^!_{n+\ell}\Lambda_{n+\ell}\bigr)_{y_0} \]
  of $\F_p$-vector spaces is $0$.
\end{lem}
\begin{proof}
  For all $m \ge 0$, let $Y_m$ be the coarse space of $\cY_m$ and $\bar{\pi}_m \colon Y_m \to Y_0$ be the map induced by $\pi_m$.
  For any $y_m \in \cY_m(k)$, the corresponding point in $Y_m(k)$ will be denoted by $\bar{y}_m$.
  Let $\cY^\wedge_{m,y_m} \colonequals \cY_m \times_{Y_m} \Spec \cO^\wedge_{Y_m,\bar{y}_m}$ be the completion of $\cY_m$ at $y_m$.
  In the previous notation, we have $\cY^\wedge_{m,y_m} = \left[V_m/H_m\right]$, and similarly $\cU_m \times_{\cY_m} \cY^\wedge_{m,y_m} = \left[U_m/H_m\right]$ and $\cD_m \times_{\cY_m} \cY^\wedge_{m,y_m} = \left[Z_m/H_m\right]$ (cf.\ the proof of \cite[Thm.~11.3.1]{MR3495343}).

  The maps $\bar{\pi}_m \colon Y_m \to Y_0$ are still quasi-finite and proper, hence finite.
  Thus, by the theorem on formal functions 
  \[ Y^\wedge_{0,\bar{y}_0} \times_{Y_0} Y_m \simeq \Spec \bigl(\bar{\pi}_*\cO_{Y_m}\bigr)^\wedge_{\bar{y}_m} \simeq \bigsqcup_{\bar{\pi}_m(\bar{y}_m)=\bar{y}_0}Y^\wedge_{m,\bar{y}_m}. \]
  As $\cY^\wedge_{0,y_0} \times_{\cY_0} \cY_m \simeq \bigl(Y^\wedge_{0,\bar{y}_0} \times_{Y_0} Y_m\bigr) \times_{Y_m} \cY_m$, the diagram
  \[ \begin{tikzcd}
      \bigsqcup_{\pi_m(y_m)=y_0}\cY^\wedge_{m,y_m} \arrow[r] \arrow[d] & \cY_m \arrow[d] \\
      \cY^\wedge_{0,y_0} \arrow[r] & \cY_0
  \end{tikzcd} \]
  is cartesian.

  Since $y_0 \colon \Spec k \to \cY_0$ factors through $\cY^\wedge_{0,y_0}$, Lemma~\ref{lem:base-change} shows that the $\cH^\ell\bigl(\R\pi_{m,*}i_{m,*}\R i^!_m\Lambda_m\bigr)_{y_0}$ as well as the morphisms between them may be computed on completions.
  The statement therefore follows from Lemma~\ref{lem:local-ramification} and exactness of the stalk functor.
\end{proof}
\begin{lem}\label{lem:base-change}
  Let $f \colon \cW \to \cY$ be a morphism of finite type between noetherian Deligne--Mumford stacks over $k$.
  Assume the coarse moduli space $Y$ of $\cY$ is excellent.
  Let $y \in \cY(k)$, giving rise to the cartesian square
  \[ \begin{tikzcd}
      \cW \times_\cY \cY^\wedge_{y} \arrow[r,"h'"] \arrow[d,"\hat{f}"] & \cW \arrow[d,"f"] \\
      \cY^\wedge_{y} \arrow[r,"h"] & \cY.
  \end{tikzcd} \]
  Let $F$ be an \'etale abelian torsion sheaf on $\cW$ whose torsion order is invertible in $k$.
  Then
  \begin{enumerate}[label={\upshape(\roman*)}]
    \item\label{lem:base-change-*} $h^*\R f_*F \simeq \R\hat{f}_*h^{\prime,*}F$
    \item\label{lem:base-change-!} If $f$ is a closed immersion, $h^*f_*\R f^! F \simeq \hat{f}_*\R\hat{f}^!h^* F$.
  \end{enumerate}
\end{lem}
\begin{proof}
  \ref{lem:base-change-*}.
  By N\'eron--Popescu desingularization \cite{MR868439} and excellence of $Y$, the natural map $Y^\wedge_{y} \to Y$ is a cofiltered limit of smooth morphisms.
  Taking fiber products with $\cY$, we see that $h \colon \cY^\wedge_{y} \to \cY$ is a cofiltered limit of smooth morphisms $h_\nu \colon \cY_\nu \to \cY$.
  For each such morphism, we have the following fiber square:
  \[ \begin{tikzcd}
      \cW \times_\cY \cY_\nu \arrow[r,"h'_\nu"] \arrow[d,"f_\nu"] & \cW \arrow[d,"f"] \\
      \cY_\nu \arrow[r,"h_\nu"] & \cY.
  \end{tikzcd} \]
  Let further $q_\nu \colon \cY^\wedge_{y} \to \cY_\nu$ denote the canonical morphisms.
  By smooth base change,
  \[ h^*\R f_*F \simeq \hocolim q^*_\nu h^*_\nu\R f_*F \simeq \hocolim q^*_\nu\R f_{\nu,*}h^{\prime,*}_\nu F \simeq \R\hat{f}_*h^{\prime,*}F. \]

  \ref{lem:base-change-!}.
  Let $j \colon \cU \hookrightarrow \cY$ be the complement of $\cW$ with its canonical open substack structure.
  By \ref{lem:base-change-*}, $h^*\R j_*j^* F \simeq \R\hat{\jmath}_*\hat{\jmath}^*h^* F$.
  Using the exact triangles
  \[ h^*f_*\R f^!F \to h^*F \to h^*\R j_*j^*F \to \quad \text{and} \quad \hat{f}_*\R\hat{f}^!h^*F \to h^*F \to \R\hat{\jmath}_*\hat{\jmath}^*h^* F \to, \]
  we see that likewise $h^*f_*\R f^! F \simeq \hat{f}_*\R\hat{f}^!h^*F$.
\end{proof}
\begin{proof}[Proof of Theorem~\ref{thm:tower-cohomology}]
  The projective system $\cY_n$ gives rise to the distinguished triangle
  \[ \hocolim_n \R\pi_{n,*} i_{n,*} \R i^!_n \Lambda_n \to \hocolim_n \R\pi_{n,*} \Lambda_n \to \hocolim_n \R\pi_{n,*} \R j_{n,*} j^*_n \Lambda_n \to. \]
  It suffices to show that $\cH^\ell(\hocolim_n \R\pi_{n,*} i_{n,*} \R i^!_n \Lambda_n) = 0$ for all $\ell$.
  This can be checked on the stalks at all finite type points $y_0 \in \cY_0(k)$.
  By \cite[\href{https://stacks.math.columbia.edu/tag/0CRK}{Lem.~0CRK}]{stacks-project} and the cocontinuity of pullback functors, $\cH^\ell(\hocolim_n \R\pi_{n,*} i_{n,*} \R i^!_n \Lambda_n)_{y_0} \simeq \colim_n \cH^\ell\bigl(\R\pi_{n,*}i_{n,*}\R i^!_n \Lambda_n\bigr)_{y_0}$, so that the assertion follows from Lemma~\ref{lem:global-ramification}.
\end{proof}
Inspired by Theorem~\ref{thm:perfectoid-Artin}, one might try to generalize Theorem~\ref{thm:tower-cohomology} and ask whether for an arbitrary constructible sheaf $F_0$ on $\cY_0$ with pullbacks $F_n$ to $\cY_n$, the natural map
\[ \colim_n \Hh^i(\cY_n,F_n) \to \colim_n \Hh^i(\cU_n,F_n) \]
is still an isomorphism for all $i \ge 0$.
Any such hope is quickly shattered by the next example.
\begin{exmp}
  Let $\cY_n \colonequals \PP^1_k$ for all $n \in \Z_{\ge 0}$, with the transition maps given by the mock Frobenii from Example~\ref{exmp:proj-perfectoid} and hence $p$-ramified as in Example~\ref{exmp:coordinate-hyperplanes}.
  Let $i_0 \colon \Spec k \hookrightarrow \cY_0$ be the inclusion of the closed point $[0:1] \in \PP^1_k$.
  The pullback of $i_0$ to $\cY_n$ is the closed immersion $i_n \colon \Spec k[x]/\bigl(x^{p^n}\bigr) \hookrightarrow \PP^1_k$ corresponding to the $p^n$-th infinitesimal thickening of $[0:1]$ in $\PP^1_k$.
  For $F_0 \colonequals i_{0,*}\F_p$, we have $F_n = i_{n,*}\F_p$.
  Let $\cU_n$ be the complement of $[0:1]$ in $\cY_n$.
  Then by the topological invariance of the \'etale site,
  \[ \colim_n \Hh^0(\cY_n,F_n) \simeq \colim_n \Hh^0\bigl(\Spec k[x]/\bigl(x^{p^n}\bigr),\F_p\bigr) \simeq \colim_n \Hh^0(\Spec k,\F_p) = \F_p, \]
  whereas $\restr{F_n}{\cU_n} \simeq 0$ and thus $\colim_n \Hh^0(\cU_n,F_n) \simeq 0$.
  Hence, the map from Theorem~\ref{thm:tower-cohomology} is not an isomorphism in this case.
\end{exmp}

\section{Moduli spaces of curves}\label{sect:moduli-curves}

\subsection{The moduli problems}

We review different moduli spaces of curves from the literature, in part to fix some notation.
Let $k$ be an algebraically closed field.
Fix an integer $g \ge 2$.
We begin with a stack-theoretic version of stable curves, which in this generality is taken from \cite[Def.~2.1, Prop.~2.3]{MR2786662}.
\begin{defn}\label{defn:twisted-curve}
  A \emph{twisted curve} is a flat, proper, tame algebraic stack $\cC \to S$ such that each geometric fiber $\cC_{\bar{s}}$ satisfies the following conditions:
  \begin{enumerate}[label={\upshape(\roman*)}]
    \item $\cC_{\bar{s}}$ is purely $1$-dimensional and connected;
    \item The coarse moduli space $C_{\bar{s}}$ is a nodal curve;
    \item The natural map $\cC_{\bar{s}} \to C_{\bar{s}}$ is an isomorphism over the smooth locus of $C_{\bar{s}}$;
    \item For the strictly henselian local ring $\cO^{\sh}_{C_{\bar{s}},x}$ at a node $x \in C_{\bar{s}}$, there is $m \in \Z_{>0}$ such that
      \[ \cC^{\sh}_{\bar{s},x} \colonequals \cC_{\bar{s}} \times_{C_{\bar{s}}} \Spec \cO^{\sh}_{C_{\bar{s}},x} \simeq \left[ \Spec (k[z,w]/(zw))^{\sh}/\mu_m \right], \]
      where $\zeta \in \mu_m$ acts on $\Spec (k[z,w]/(zw))^{\sh}$ by $z \mapsto \zeta z$ and $w \mapsto \zeta^{-1} w$.
  \end{enumerate}
\end{defn}
\begin{defn}[{\cite[Def.~6.1.1]{MR2007376}}]
  Let $m \in \Z_{>0}$.
  A \emph{pre-level-$m$ curve} is a twisted curve $\cC \to S$ whose coarse moduli space is a stable curve and whose geometric fibers have trivial stabilizer at each separating node and stabilizer $\mu_m$ at each non-separating node.
\end{defn}
Pre-level-$m$ curves without non-trivial stabilizers are exactly the curves of compact type.
\begin{defn}
  A stable curve $C \to S$ is \emph{of compact type} if each geometric fiber $C_{\bar{s}}$ satisfies one of the following equivalent conditions:
  \begin{enumerate}[label={\upshape(\roman*)}]
    \item The Jacobian $J(C_{\bar{s}})$ is compact;
    \item All nodes of $C_{\bar{s}}$ are separating;
    \item The dual graph of $C_{\bar{s}}$ is a tree.
  \end{enumerate}
\end{defn}
From now on, we will assume that $m \in \Z_{>0}$ is invertible in $k$.
In that case, we can endow pre-level-$m$ curves with level structures.
\begin{defn}[{\cite[\S~6]{MR2007376}}]
  Let $f\colon \cC \to S$ be a pre-level-$m$ curve of genus $g$.
  A \emph{full level-$m$ structure} on $\cC$ is the choice of a symplectic isomorphism of local systems
  \[ \Rr^1f_*(\Z/m\Z) \xrightarrow{\sim} \underline{(\Z/m\Z)^{2g}}. \]
\end{defn}
We consider the following smooth Deligne--Mumford stacks:
\begin{center}
  \begin{tabularx}{\linewidth}{cX}
    $\cM_g$ & moduli stack of smooth curves of genus $g$ over $k$ \\
    $\cM^{\cc}_g$ & moduli stack of curves of compact type of genus $g$ over $k$ \\
    $\overline{\cM_g}$ & moduli stack of stable curves of genus $g$ over $k$ \\
    $\cM_g[m]$ & moduli stack of smooth curves of genus $g$ over $k$ with full level-$m$ structure \\
    $\cM^{\cc}_g[m]$ & moduli stack of curves of compact type of genus $g$ over $k$ with full level-$m$ structure \\
    $\overline{\cM_g}[m]$ & moduli stack of pre-level-$m$ curves of genus $g$ over $k$ with full level-$m$ structure.
  \end{tabularx}
\end{center}
We use Roman letters (i.e., $M_g$, $M^{\cc}_g$, etc.) for the corresponding coarse spaces.
By covering space theory, $\overline{\cM_g}[m]$ can be identified with the stack of connected $(\Z/m\Z)^{2g}$-torsors over twisted curves with stable coarse moduli space, rigidified along $(\Z/m\Z)^{2g}$ \cite[Thm.~6.2.4]{MR2007376}.
There are natural open embeddings $\cM_g \subset \cM^{\cc}_g \subset \overline{\cM_g}$.
The natural covers $\cM_g[m] \to \cM_g$ and $\cM^{\cc}_g[m] \to \cM^{\cc}_g$ are finite \'etale.
Although the cover $\overline{\cM_g}[m] \to \overline{\cM_g}$ is flat, proper, and quasi-finite \cite[Cor.~3.0.5]{MR2007376}, it is not representable \cite[Rem.~5.2.4.(b)]{MR2007376} and highly ramified over $\overline{\cM_g} \smallsetminus \cM^{\cc}_g$ (see Lemma~\ref{lem:Mg-p-ramified}).

\subsection{Local structure of $\overline{\cM_g}[m]$}\label{subsect:local-structure}

Next, we summarize the description of the complete, local picture from \S~\ref{sect:boundary-vanishing} for a given point $y \in \overline{\cM_g}[m](k)$ in a sequence of well-known lemmas.
Most of what follows can be found in \cite{MR2007376}.

Fix $m \in \Z_{>0}$ invertible in $k$.
Set $G \colonequals (\Z/m\Z)^{2g}$.
The datum parametrized by $y$ is a pre-level-$m$ curve $\cC$ over $k$ of genus $g$ together with a full level-$m$ structure $\Hh^1(\cC,\Z/m\Z) \xrightarrow{\sim} G$, or equivalently a connected $G$-torsor $P \to \cC$.
Let $C$ be the coarse space of $\cC$.
Assume that $C$ has non-separating nodes $x_1,\dotsc,x_\gamma$ and separating nodes $x_{\gamma + 1},\dotsc,x_\delta$.

Since $\overline{\cM_g}[m]$ is Deligne--Mumford and $P \to \cC$ is \'etale, $\cO^\wedge_{\overline{\cM_g}[m],y}$ is the universal deformation ring of $\cC$.
Let $\fC \to \Spec \cO^\wedge_{\overline{\cM_g}[m],y}$ be the universal curve.
\begin{lem}\label{lem:description-deformation-ring}
  \hangindent\leftmargini
  \textup{(i)}\hskip\labelsep The completed local ring $\cO^\wedge_{\overline{\cM_g}[m],y}$ is regular and of dimension $3g-3$.
  \begin{enumerate}[label={\upshape(\roman*)}]
    \setcounter{enumi}{1}
    \item\label{lem:description-deformation-ring-local} There is an isomorphism $\cO^\wedge_{\overline{\cM_g}[m],y} \simeq k\llbracket t_1,\dotsc,t_{3g-3}\rrbracket$ such that the locus over which the node $x_i$ persists in $\mathfrak{C}$ is $\{t_i = 0\}$.
  \end{enumerate}
\end{lem}
\begin{proof}
  (i).
  Since $\cO^\wedge_{\overline{\cM_g}[m],y}$ is the universal deformation ring of $\cC$, both properties follow from the usual analysis of the deformation theory of $\cC$;
  see \cite[\S~3]{MR2007376}.

  \ref{lem:description-deformation-ring-local}.
  The local-to-global Ext spectral sequence decomposes the tangent space $\Ext^1(\Omega^1_\cC,\cO_\cC)$ of $\cO^\wedge_{\overline{\cM_g}[m],y}$ via the short exact sequence
  \begin{equation}\label{eqn:local-global-principle}
    0 \to \Hh^1(\cC,\cT_\cC) \to \Ext^1\bigl(\Omega^1_\cC,\cO_\cC\bigr) \to \prod^\delta_{i=1} \Ext^1\bigl(\Omega^{1,\wedge}_{\cC,x_i},\cO^\wedge_{\cC,x_i}\bigr) \to 0
  \end{equation}
  into a ``global contribution'' from the left and a ``local contribution'' from the right term;
  cf.\ \cite[Prop.~1.5]{MR0262240}.
  To analyze the local part, note that for each node $x_i$, the completed local ring $\cO^\wedge_{\cC,x_i} \simeq k\llbracket z,w \rrbracket / (zw)$ of $\cC$ at $x_i$ admits the universal formal deformation $k \llbracket z,w,t_i \rrbracket / (zw - t_i)$.
  Thus, the universal deformation ring of the node is given by $k \llbracket t_i \rrbracket$ and its tangent space $\Ext^1\bigl(\Omega^{1,\wedge}_{\cC,x_i},\cO^\wedge_{\cC,x_i}\bigr)$ is of dimension $1$.

  Since the $\cO^\wedge_{\fC,x_i}$ are deformations of $\cO^\wedge_{\cC,x_i}$, there are natural morphisms $\varphi_i \colon k \llbracket t_i \rrbracket \to \cO^\wedge_{\overline{\cM_g}[m],y}$.
  Their completed tensor product
  \[ \varphi_1 \widehat{\otimes} \dotsb \widehat{\otimes} \varphi_\delta \colon k \llbracket t_1,\dotsc,t_\delta \rrbracket \to \cO^\wedge_{\overline{\cM_g}[m],y} \]
  is formally smooth because its tangent map from (\ref{eqn:local-global-principle}) is surjective.
  Therefore, we can choose an isomorphism $\cO^\wedge_{\overline{\cM_g}[m],y} \simeq \Spec k\llbracket t_1,\dotsc,t_{3g-3}\rrbracket$ so that $\varphi_i$ is identified with the inclusion of the $i$-th coordinate for all $1 \le i \le \delta$.
  In particular, the node $x_i$ persists over $\{t_i = 0\}$ in $\fC$.
\end{proof}
\begin{cor}\label{cor:non-ct-nc}
  The locus $\cD_m \colonequals \overline{\cM_g}[m] \smallsetminus \cM^{\cc}_g[m]$ of curves not of compact type with its reduced closed substack structure is a normal crossings divisor in $\overline{\cM_g}[m]$ whose pullback to $\cO^\wedge_{\overline{\cM_g}[m],y}$ is given by $\{ t_1\dotsm t_\gamma = 0 \}$ under the isomorphism from Lemma~\ref{lem:description-deformation-ring}.\ref{lem:description-deformation-ring-local}.
\end{cor}
We refer to Definition~\ref{defn:nc-divisor} for the notion of a normal crossings divisor on a smooth Deligne--Mumford stack.
\begin{lem}\label{lem:description-automorphisms}
  When $m \ge 3$, the automorphism group $\Aut^G(P \to \cC)$ of the $G$-torsor $P \to \cC$ is $G \times H_y$, where $H_y$ is a subgroup of $\Aut_C(\cC)$.
\end{lem}
\begin{proof}
  As $m \geq 3$, a lemma of Serre shows that $\Aut^G(P \to \cC)$ is contained in the group $\Aut_C(P \to \cC)$ of automorphism of $P \to \cC$ which act trivially on the underlying coarse space $C$ \cite[Lem.~7.2.1]{MR2007376}.
  By \cite[Lem.~7.3.3]{MR2007376}, the sequence
  \[ 1 \to \Aut_\cC(P) \to \Aut_C(P \to \cC) \to \Aut_C(\cC) \to 1 \]
  is split exact.
  Since $\Aut^G(P \to \cC)$ is the centralizer of $\Aut_\cC(P) \simeq G$ in $\Aut_C(P \to \cC)$ and $G$ is abelian, the statement follows with $H_y \colonequals \ker(\Aut_C(\cC) \to \Aut(G))$.
\end{proof}
\begin{cor}
  The completion of $\overline{\cM_g}[m]$ at $y$ for $m \ge 3$ is given by
  \[ \overline{\cM_g}[m]^\wedge_y \colonequals \overline{\cM_g}[m] \times_{\overline{M_g}[m]} \Spec \cO^\wedge_{\overline{M_g}[m],y} \simeq \left[ \Spec k\llbracket t_1,\dotsc,t_{3g-3}\rrbracket / H_y \right], \]
  with the open substacks $\cM^{\cc}_g[m] \times_{\overline{\cM_g}[m]} \overline{\cM_g}[m]^\wedge_y$ and $\cM_g[m] \times_{\overline{\cM_g}[m]} \overline{\cM_g}[m]^\wedge_y$ corresponding to the loci $\left[ (\Spec k\llbracket t_1,\dotsc,t_{3g-3}\rrbracket_{t_1\dotsm t_\gamma}) / H_y \right]$ and $\left[ (\Spec k\llbracket t_1,\dotsc,t_{3g-3}\rrbracket_{t_1\dotsm t_\delta}) / H_y \right]$, respectively.
\end{cor}
\begin{proof}
  As in Lemma~\ref{lem:global-ramification}, the proof of \cite[Thm.~11.3.1]{MR3495343} shows $\overline{\cM_g}[m]^\wedge_y \simeq \Bigl[ \Spec \cO^\wedge_{\overline{\cM_g}[m],y} / \Aut(y) \Bigr]$.
  Since $\overline{\cM_g}[m]$ is rigidified along $G$, we have $\Aut(y) \simeq H_y$ by Lemma~\ref{lem:description-automorphisms}.
  Now use Lemma~\ref{lem:description-deformation-ring}.
\end{proof}
We will need a more concrete description of $\Aut_C(\cC)$ in Example~\ref{exmp:Mumford-compactification}.
\begin{lem}[{\cite[\S~7.1]{MR2007376}}]\label{lem:description-AutC}
  There is an isomorphism
  \[ \Aut_C(\cC) \simeq \prod^\gamma_{i=1} \mu_m; \]
  if $U \colonequals \Spec(k[z,w]/(zw))^{\sh}$ is an \'etale atlas for the strict henselization $\cC^{\sh}_{x_i}$ at a non-separating node $x_i$ as in Definition~\ref{defn:twisted-curve}, the automorphisms from the $i$-th factor act trivially on $\cC \smallsetminus \{x_i\}$, and on $\cC^{\sh}_{x_i}$ as
\[ \Aut_{C^{\sh}_{x_i}}\cC^{\sh}_{x_i} \simeq (\Aut_{U/\mu_m} U) / (\Aut_{\left[ U/\mu_m \right]} U) \simeq \mu^2_m / \mu_m \simeq \mu_m. \]
\end{lem}
\begin{exmp}\label{exmp:description-Hy}
  Lemma~\ref{lem:description-AutC} leads to a hands-on description of $H_y$ when $k=\C$.
  Let $\cC_\Delta \to \Delta$ be a deformation of $\cC$ over a small polydisc $\Delta$ with smooth general fiber $C_\eta$.
  For $1 \leq i \leq \gamma$, let $c_i$ be the vanishing cycle of $C_\eta$ corresponding to the node $x_i$ of $C$.
  A full level-$m$ structure on $\cC$ induces an isomorphism $G \simeq \Hh_1(C_\eta,\Z/m\Z)$.
  With this identification, we can pick $(\zeta_1,\dotsc,\zeta_\gamma) \in \prod^\gamma_{i=1}\mu_m \simeq \Aut_C(\cC)$ such that $\zeta_i$ acts on $G$ via the Dehn twist
  \[ \Hh_1(C_\eta,\Z/m\Z) \to \Hh_1(C_\eta,\Z/m\Z), \quad \alpha \mapsto \alpha + (\alpha \cdot c_i)c_i \]
  \cite[Lem.~7.3.3]{MR2007376}.
  Since the intersection pairing on $\Hh_1(C_\eta,\Z)$ is unimodular and the vanishing cycles form part of a basis, $H_y \simeq \ker\bigl(\bigoplus_i \Hh_1(c_i,\Z/m\Z) \to \Hh_1(C_\eta,\Z/m\Z)\bigr)$.

  From the long exact sequence in homology for the pair $\bigl(C,\bigcup_i c_i\bigr)$, we see further that $H_y \simeq \im \bigl(\Hh_2(C_\eta,\bigcup_i c_i;\Z/m\Z) \to \bigoplus_i \Hh_1(c_i,\Z/m\Z)\bigr)$.
  Let $\nu$ be the number of irreducible components of $C$.
  Then $C_\eta \smallsetminus \bigl(\bigcup_i c_i \bigr)$ has $\nu - (\delta-\gamma)$ connected components $C_1,\dotsc,C_{\nu-\delta+\gamma}$.
  In particular, $\Hh_2(C_\eta,\bigcup_i c_i;\Z/m\Z) \simeq \bigoplus_j \Z/m\Z \cdot [C_j]$, where $[C_j]$ denotes the fundamental class of the closure of $C_j$ in $C_\eta$.

  Let $\Gamma$ be the dual graph of $C$.
  Let $\Gamma'$ be the graph obtained from $\Gamma$ by contracting all bridges.
  Its vertices correspond to the connected components $C_j$ and its edges to the non-separating nodes $x_i$.
  After fixing compatible orientations for the $c_i$ and $\Gamma'$, we can identify the complex $\Hh_2(C_\eta,\bigcup_i c_i;\Z/m\Z) \to \bigoplus_i \Hh_1(c_i,\Z/m\Z)$ with the cellular cochain complex $\Cc^0(\Gamma',\Z/m\Z) \to \Cc^1(\Gamma',\Z/m\Z)$.
  In particular, $H_y$ is a free $\Z/m\Z$-module of rank $\gamma - b_1(\Gamma') = \gamma - b_1(\Gamma) = \gamma-\delta+\nu-1$, as can also be seen directly from the long exact homology sequence.
\end{exmp}

\subsection{Towers of moduli spaces}\label{subsect:towers}

Assume $\charac k \neq p$.
For each $n \in \Z_{\geq 0}$, set $G_n \colonequals (\Z/p^n\Z)^{2g}$.
We have a system of natural flat, proper, quasi-finite maps
\[ \dotsb \to \overline{\cM_g}[p^{n+1}] \to \overline{\cM_g}[p^n] \to \dotsb, \]
which are given on $S$-valued points by
\[ (P_{n+1} \to \cC_{n+1} = \left[ P_{n+1}/G_{n+1}\right] \to S) \mapsto (P_n \colonequals P_{n+1}/(p^n\Z/p^{n+1}\Z)^{2g} \to \cC_n \colonequals \left[ P_n/G_n \right] \to S). \]
The next lemma, which uses the notion of $p$-ramified morphisms from Definition~\ref{defn:m-ramified}, will allow us to apply Theorem~\ref{thm:tower-cohomology} to the moduli spaces of curves from above.
\begin{lem}[{\cite[Thm.~5.1.5]{MR2920693} or \cite[Rmk.~1.11]{MR2309994}}]\label{lem:Mg-p-ramified}
  The maps $\overline{\cM_g}[p^{n+1}] \to \overline{\cM_g}[p^n]$ are $p$-ramified over the normal crossings divisor $\cD_{p^n}$ from Corollary~\ref{cor:non-ct-nc}.
\end{lem}
The proof relies on the description of the completed local rings from Lemma~\ref{lem:description-deformation-ring}.
The ``local contributions'' of the nodes to a map $\cO^\wedge_{\overline{\cM_g}[p^n],y_n} \to \cO^\wedge_{\overline{\cM_g}[p^{n+1}],y_{n+1}}$ can be computed using the universal formal deformations of the nodes described in the proof of Lemma~\ref{lem:description-deformation-ring}.\ref{lem:description-deformation-ring-local};
this yields the desired ramification at every node with non-trivial stabilizer.

\section{Vanishing for moduli spaces of curves}\label{sect:moduli-curves-vanishing}

In this section, we discuss how to apply the vanshing theorems from \S~\ref{sect:perverse-vanishing}, \S~\ref{sect:perfectoid-Artin}, and \S~\ref{sect:boundary-vanishing} to the moduli spaces of curves reviewed in \S~\ref{sect:moduli-curves}.
We retain the conventions and notation from \S~\ref{sect:moduli-curves}.

\subsection{The Torelli morphism}\label{subsect:Torelli}

The Torelli morphism is the functor
\[ t_g \colon \cM^{\cc}_g \to \cA_g \]
which sends a curve of compact type to its (principally polarized) Jacobian.
A detailed account of the construction of $t_g$ is, for example, given in \cite{landesman2019properness}.
For $m \in \Z_{>0}$ invertible in $k$, let $\cA_g[m]$ be the moduli space of principally polarized abelian varieties of dimension $g$ over $k$ with full level-$m$ structure.
Since full level-$m$ structures on curves of compact type correspond to full level-$m$ structures on their Jacobians, the stack $\cM^{\cc}_g[m]$ from \S~\ref{sect:moduli-curves} fits into a pullback square
  \[ \begin{tikzcd}
      \cM^{\cc}_g[m] \arrow[r,"t^m_g"] \arrow[d,"\pi_m"] & \cA_g[m] \arrow[d,"\rho_m"] \\
      \cM^{\cc}_g \arrow[r,"t_g"] & \cA_g
  \end{tikzcd} \]
and we obtain a Torelli map $t^m_g \colon \cM^{\cc}_g[m] \to \cA_g[m]$ at level $m$.
The scheme-theoretic image of $t^m_g$ (\cite[\href{https://stacks.math.columbia.edu/tag/0CMH}{\S~0CMH}]{stacks-project}) is called the Torelli locus $\cT_g[m] \subseteq \cA_g[m]$.
Set $\cT_g \colonequals \cT_g[0]$ and $\cT_1 \colonequals \cA_1$.

Since $t^m_g$ is only generically finite, we cannot expect that $\colim_n \Hh^i_{\et}(\cM^{\cc}_g[p^n],\F_p) = 0$ for all $i > \dim \cM^{\cc}_g = 3g-3$, even though $\overline{\cA_g}[p^n]$ is perfectoid at infinite level;
cf.\ Example~\ref{exmp:Ag-tor} and Example~\ref{exmp:no-vanishing-generically-finite}.
In this subsection, we instead apply the results from \S~\ref{sect:perverse-vanishing} to $t^m_g$.
First, we recall two well-known statements.
\begin{lem}\label{lem:Torelli-proper}
  The morphism $t^m_g$ is representable and proper.
\end{lem}
\begin{proof}
  Representable and proper morphisms are stable under base change, so it suffices to show both properties for $m = 0$.
  To prove that $t_g$ is representable, we only have to see that for every curve $C$ of compact type over an algebraically closed field, the induced group homomorphism $\Aut(C) \to \Aut\bigl(J(C)\bigr)$ is injective;
  cf.\ e.g.\ \cite[Lem.~4.4.3]{MR1862797}.
  This is \cite[Thm.~1.13]{MR0262240}.

  Since $\cM^{\cc}_g$ is of finite type and separated and $\cA_g$ is locally noetherian with separated diagonal, properness follows from the existence part of the valuative criterion for all discrete valuation rings $V$ with fraction field $K$.
  Let $\fA \in \cA_g(V)$ (suppressing principal polarizations from the notation) and $C \in \cM^{\cc}_g(K)$ such that $\fA_K \simeq J(C)$.
  The stable reduction theorem for curves produces an extension $K \subseteq K'$ and a stable curve $\fC$ over a valuation ring $V' \subset K'$ dominating $V$ with residue field $\kappa$ whose generic fiber is isomorphic to $C_{K'}$.
  By Weil's extension theorem for rational maps into group schemes, $\fA_{V'}$ is the N\'eron model of $\fA_{K'}$.
  In particular, a theorem of Raynaud \cite[Thm.~8.2.1]{MR282993} (cf.\ also \cite[Thm.~2.5]{MR0262240}) shows that $J(\fC_\kappa) \simeq \fA_\kappa$, so $\fC_\kappa$ is of compact type and $\fC \in \cM^{\cc}_g(V')$ is the desired lift of $\fA_{V'}$.
\end{proof}
\begin{lem}[Torelli for compact type curves]\label{lem:ct-Torelli}
  Let $C$ and $D$ be two curves of compact type of genus $g$ over $k$.
  Let $C_1,\dotsc,C_\delta$ and $D_1,\dotsc,D_\varepsilon$ be their irreducible components of positive genus.
  If $t_g(C) \simeq t_g(D)$, then $\delta = \varepsilon$ and $C_i \simeq D_{\sigma(i)}$ for some permutation $\sigma \in S_\delta$.
\end{lem}
\begin{proof}
  We have $\prod^\delta_{i=1} J(C_i) \simeq J(C) \simeq J(D) \simeq \prod^\varepsilon_{j=1} J(D_j)$ as principally polarized abelian varieties.
  Since a principally polarized abelian variety over an algebraically closed field uniquely decomposes into an unordered product of indecomposable principally polarized abelian varieties (\cite[Lem.~3.20]{MR0302652}, \cite[Cor.~2]{MR1411055}) and the polarization for each $C_i$ and $D_j$ is induced by its irreducible theta divisor, we must have $\delta = \varepsilon$ and $J(C_i) \simeq J(D_{\sigma(i)})$ for some $\sigma \in S_\delta$.
  The statement now follows from the usual Torelli theorem for smooth, projective curves.
\end{proof}
In other words, $t_g(C)$ determines the irreducible components of $C$ of positive genus, but does not depend on the position of the nodes or the number of rational irreducible components.
A dimension analysis (keeping in mind the dimensions of the automorphism groups of curves of genus $0$ and $1$) then shows the following statement.
\begin{lem}[{\cite[Prop.~5.2.1]{MR2825165}}]\label{lem:fiber-dimension}
  Let $C$ be a curve of compact type of genus $g$ over $k$.
  Let $\delta$ and $\delta_1$ be the number of irreducible components of $C$ of positive genus and genus $1$, respectively.
  Then
  \[ \dim \Bigl(\bigl(\cM^{\cc}_g\bigr)_{t_g(C)}\Bigr) = 2\delta - \delta_1 - 2. \]
\end{lem}
Next, we describe a stratification on $\abs{\cT_g}$ that is suitable to determine the defect of semi-smallness of $t_g$ (see Definition~\ref{defn:defect}).
\begin{lem}\label{lem:product-finite}
  Let $\lambda = (\lambda_1,\dotsc,\lambda_\delta)$ be a partition of $g$, that is, $\lambda_1 \ge \lambda_2 \ge \dotsb \ge \lambda_\delta$ and $\sum^\delta_{i=1}\lambda_i = g$.
  Then the product morphism $\xi_\lambda \colon \prod^\delta_{i=1} \cT_{\lambda_i} \to \cT_g$ is finite.
\end{lem}
\begin{proof}
  We show that $\xi_\lambda$ is representable, proper, and locally quasi-finite.
  Representability follows from the faithfulness of the product functor \cite[\href{https://stacks.math.columbia.edu/tag/04Y5}{Lem.~04Y5}]{stacks-project}.
  Since the domain is separated and of finite type and the target has separated diagonal, $\xi_\lambda$ is separated and of finite type.

  For universal closedness, we can again verify the existence part of the valuative criterion for all discrete valuation rings $V$ with fraction field $K$ because $\cT_g$ is locally noetherian.
  Let $\fA \in \cT_g(V)$ and $A_i \in \cT_{\lambda_i}(K)$ such that $\fA_K \simeq \prod A_i$.
  The N\'eron--Ogg--Shafarevich criterion \cite[Thm.~1]{MR236190} and the compatibility of Tate modules with products show that the $A_i$ have good reduction.
  Furthermore, their principal polarizations extend to the integral models;
  cf.\ e.g.\ the argument in \cite[p.~6]{MR1083353}.
  Since $\cT_{\lambda_i} \subseteq \cA_{\lambda_i}$ is closed, we thus obtain $\fA_i \in \cT_{\lambda_i}(V)$ with generic fiber isomorphic to $A_i$.
  Separatedness of $\cA_g$ gives $\fA \simeq \prod \fA_i$ as wanted.

  Finally, to prove that $\xi_\lambda$ is locally quasi-finite, we can check that for every morphism $\Spec(K) \to \cT_g$ from a field $K$, the space $\abs{\Spec{K} \times_{\cT_g} \prod \cT_{\lambda_i}}$ is discrete \cite[\href{https://stacks.math.columbia.edu/tag/06UA}{Lem.~06UA}]{stacks-project}.
  Since the base change morphism $\abs{\Spec{\overline{K}} \times_{\cT_g} \prod \cT_{\lambda_i}} \to \abs{\Spec{K} \times_{\cT_g} \prod \cT_{\lambda_i}}$ is surjective and integral (hence closed), it suffices to show that $\abs{\Spec{\overline{K}} \times_{\cT_g} \prod \cT_{\lambda_i}}$ is discrete.
  This follows from the uniqueness of the indecomposable factors of principally polarized abelian varieties used in the proof of Lemma~\ref{lem:ct-Torelli}.
\end{proof}
Below, we will use the partial order on the set of integer partitions given by refinement, as introduced in \cite[Ex.~I.8.10]{MR598630}.
\begin{defn}
  Let $\lambda = (\lambda_1,\dotsc,\lambda_\delta)$ and $\mu = (\mu_1,\dotsc,\mu_\varepsilon)$ be two partitions of $g$.
  Then $\mu$ \emph{refines} $\lambda$, or $\mu \le \lambda$, if there exists a set partition $\{ 1,\dotsc,\varepsilon\} = I_1 \cup \dotsb \cup I_\delta$ such that $\lambda_j = \sum_{i \in I_j} \mu_i$ for all $1 \le j \le \delta$.
\end{defn}
\begin{defn}
  For each partition $\lambda$ of $g$, let $\cT_\lambda$ be the scheme-theoretic image of $\xi_\lambda$.
  Set
  \[ S_\lambda \colonequals \abs{\cT_\lambda} \smallsetminus \bigcup_{\mu < \lambda} \abs{\cT_\mu}. \]
\end{defn}
\begin{lem}\label{lem:Torelli-stratification}
  \begin{enumerate}[leftmargin=*,label={\upshape(\roman*)}]
    \item\label{lem:Torelli-stratification-existence} The subspaces $S_\lambda \subset \abs{\cT_g}$ associated to the partitions $\lambda$ of $g$ are locally closed and parametrize the Jacobians of curves of compact type whose geometric fibers have exactly $\delta$ irreducible components of positive genus $\lambda_i$, giving rise to a finite stratification $\abs{\cT_g} = \bigsqcup_\lambda S_\lambda$.
    \item\label{lem:Torelli-stratification-dimension} For each integer partition $\lambda$ of $g$ into $\delta$ parts, we have $\dim S_\lambda = 3g - 3\delta + \delta_1$, where $\delta_1 \colonequals \#\{ i \suchthat \lambda_i = 1 \}$.
  \end{enumerate}
\end{lem}
\begin{proof}
  \ref{lem:Torelli-stratification-existence}.
  Since $\xi_\lambda$ is closed by Lemma~\ref{lem:product-finite}, it surjects onto its scheme-theoretic image.
  Thus, the closed substack $\cT_\lambda \subseteq \cT_g$ parametrizes those families whose geometric fibers are products of $\delta$ Jacobians of compact type curves of genus $\lambda_1,\dotsc,\lambda_\delta$.
  By Lemma~\ref{lem:ct-Torelli}, the geometric fibers are the Jacobians of exactly those curves of compact type that have $\delta$ (not necessarily irreducible) components of genus $\lambda_1,\dotsc,\lambda_\delta$.
  This yields the first statement as $\abs{\cT_\mu} \subseteq \abs{\cT_\lambda}$ if $\mu \le \lambda$ and $\abs{\cT_\mu} \cap \abs{\cT_\lambda} = \varnothing$ if not.

  \ref{lem:Torelli-stratification-dimension}.
  Since $\xi_\lambda$ is finite by Lemma~\ref{lem:product-finite}, $t_{\lambda_i}$ is generically finite, and $\cT_\lambda$ is Deligne--Mumford,
  \[ \dim \abs{\cT_\lambda} = \dim \cT_\lambda = \sum^\delta_{i=1} \cT_{\lambda_i} = \sum_{\lambda_i \ge 2} (3\lambda_i - 3) + \sum_{\lambda_i =1} 1 = 3g - 3\delta + \delta_1; \]
  cf.\ \cite[\href{https://stacks.math.columbia.edu/tag/0DRK}{Rem.~0DRK}]{stacks-project}.
  The number of summands of a partition increases with refinement, hence $\dim S_\lambda = \dim \abs{\cT_\lambda} = 3g - 3\delta + \delta_1$.
\end{proof}
\begin{rem}
  The preceding arguments also show the existence of a finite stratification of $\cA_g$ by number and dimension of indecomposable principally polarized factors.
\end{rem}
\begin{prop}\label{prop:Torelli-fiber-defect}
  The Torelli morphism $t_g \colon \cM^{\cc}_g \to \cA_g$ has maximal fiber dimension $g-2$ and defect of semi-smallness $r(t_g) = \floor*{\frac{g}{2}} - 1$.
\end{prop}
\begin{proof}
  The statement about the fiber dimension follows from Lemma~\ref{lem:fiber-dimension} and Lemma~\ref{lem:component-bounds}.\ref{lem:component-bounds-fiber} below.
  For the defect of semi-smallness, we use the stratification from Lemma~\ref{lem:Torelli-stratification}.\ref{lem:Torelli-stratification-existence}.
  Let $\lambda = (\lambda_1,\dotsc,\lambda_\delta)$ be an integer partition of $g$.
  As before, set $\delta_1 \colonequals \#\{ i \suchthat \lambda_i = 1 \}$.
  The dimension of $S_\lambda$ is $3g - 3\delta + \delta_1$ by Lemma~\ref{lem:Torelli-stratification}.\ref{lem:Torelli-stratification-dimension} and the relative dimension of $t_g$ over $S_\lambda$ is $2\delta - \delta_1 - 2$ by Lemma~\ref{lem:fiber-dimension}.
  The statement about $r(t_g)$ is therefore a consequence of Lemma~\ref{lem:component-bounds}.\ref{lem:component-bounds-defect} and the equality
  \[ 2 \cdot (2\delta - \delta_1 - 2) + (3g - 3\delta + \delta_1) - (3g - 3) = \delta - \delta_1 - 1. \qedhere \]
\end{proof}
\begin{lem}\label{lem:component-bounds}
  Let $C$ be a curve of compact type of genus $g$ over $k$.
  Let $\delta$ and $\delta_1$ be the number of irreducible components of $C$ of positive genus and genus $1$, respectively.
  Then
  \begin{enumerate}[label={\upshape(\roman*)}]
    \item\label{lem:component-bounds-fiber} $2\delta - \delta_1 \le g$
    \item\label{lem:component-bounds-defect} $\delta - \delta_1 \le \floor*{\frac{g}{2}}$
  \end{enumerate}
  and both bounds are sharp.
\end{lem}
\begin{proof}
  Since both $2\delta - \delta_1$ and $\delta - \delta_1$ do not change under contraction of rational components, we may assume that $C$ does not have any rational components.
  Moreover, since both expressions do not decrease when a component of genus $g' \ge 3$ specializes to a union of a component of genus $g' - 2$ and $2$, we may further assume that all irreducible components of $C$ have genus $1$ or $2$.
  In that case, $2\delta - \delta_1 = g$ and $\delta - \delta_1$ is maximal when $C$ has $\floor*{\frac{g}{2}}$ irreducible components of genus $2$, with one component of genus $1$ if $g$ is odd, yielding the inequalities.
\end{proof}
\begin{cor}\label{cor:Torelli}
  Let $K \in \D(\cM^{\cc}_g[m])$.
  Then
  \begin{enumerate}[label={\upshape(\roman*)}]
    \item\label{cor:Torelli-perverse} $\R t^m_{g,*}K \in \pD(\cA_g[m])^{\le n+g-2}$ if $K \in \pD(\cM^{\cc}_g[m])^{\le n}$ and
    \item\label{cor:Torelli-constructible} $\R t^m_{g,*}K \in \pD(\cA_g[m])^{\le \floor*{\frac{g}{2}} - 1}$ if $K = F[3g-3]$ for some constructible sheaf $F$ of $\F_p$-modules on $\cM^{\cc}_g[m]$.
  \end{enumerate}
\end{cor}
\begin{proof}
  Since the forgetful maps $\rho_m \colon \cA_g[m] \to \cA_g$ are finite, the direct image functors $\rho_{m,*}$ are t-exact for the perverse t-structure (\cite[Cor.~2.2.6.(i)]{MR751966}) and conservative.
  Consequently, $\R t^m_{g,*}K \in \pD(\cA_g[m])^{\le \ell}$ if and only if $\rho_{m,*}\R t^m_{g,*}K \in \pD(\cA_g)^{\le \ell}$, and it suffices to show the statement when $m = 0$.
  In that case, it follows from Lemma~\ref{lem:Torelli-proper}, Proposition~\ref{prop:proper-exact}, Lemma~\ref{lem:defect-exact}, and Proposition~\ref{prop:Torelli-fiber-defect}.
\end{proof}

\subsection{Proof of the main results}\label{subsect:main-results}
From here on, we additionally assume that $k = \C$.
By fixing a (non-canonical) isomorphism $\C \simeq \C_p$, we then have all results of \S~\ref{sect:perfectoid-Artin} at our disposal.
We are now ready to prove Theorem~\ref{thm:perverse} and Theorem~\ref{thm:Mgbar}, which we restate for the convenience of the reader.
\perverse*
\begin{proof}
  Let $\rho_n \colon \cA_g[p^n] \to \cA_g$ be the natural maps forgetting the level structure.
  Since the canonical maps $\rho^*_n\R t_{g,*} K \to \R t^{p^n}_{g,*}\pi^*_n K$ are isomorphisms for all $n \ge 0$ and all $K \in \D(\cM^{\cc}_g)$ by proper base change, Corollary~\ref{cor:Torelli} reduces the assertion to showing that $\colim_n\Hh^i_{\et}(\cA_g[p^n],\rho^*_n L) = 0$ for all $L \in \pD(\cA_g)^{\le \ell}$ and all $i > \ell$.

  As in Example~\ref{exmp:Ag-tor}, we denote by $\overline{\cA_g}[p^n]$ the toroidal compactifications of $\cA_g[p^n]$ determined by a fixed smooth, projective $\GL_g(\Z)$-admissible polyhedral decomposition of the cone of positive semi-definite quadratic forms on $\R^g$ whose null space is defined over $\Q$.
  Let $\overline{\rho}_{mn} \colon \overline{\cA_g}[p^m] \to \overline{\cA_g}[p^n]$ be the natural transition maps.
  The inclusions $\iota_n \colon \cA_g[p^n] \hookrightarrow \overline{\cA_g}[p^n]$ cut out the complement of a Cartier divisor and are thus affine morphisms.
  Therefore, $\R\iota_{n,*}\rho^*_n L \in \pD(\overline{\cA_g}[p^n])^{\le \ell}$ (Theorem~\ref{thm:affine-exact}) and
  \[ \colim_n\Hh^i_{\et}(\cA_g[p^n],\rho^*_n L) \simeq \colim_n\Hh^i_{\et}(\overline{\cA_g}[p^n],\R\iota_{n,*}\rho^*_n L) = 0 \]
  for all $i > \ell$ by Example~\ref{exmp:Ag-tor} and Corollary~\ref{cor:perfectoid-Artin-perverse} applied to $K_n = \R\iota_{n,*}\rho^*_n L$ and the natural base change maps $\varphi^*_{mn} \colon \overline{\rho}^*_{mn}\R\iota_{n,*}\rho^*_n L \to \R\iota_{m,*}\rho^*_m L$.
\end{proof}
\Mgbar*
\begin{proof}
  By Corollary~\ref{cor:non-ct-nc}, the locus $\cD_{p^n} \subset \overline{\cM_g}[p^n]$ of curves not of compact type is a normal crossings divisor.
  Lemma~\ref{lem:Mg-p-ramified} shows that the transition maps $\overline{\cM_g}[p^{n+1}] \to \overline{\cM_g}[p^n]$ are $p$-ramified over $\cD_{p^n}$.
  Thus, the assertion follows from Theorem~\ref{thm:tower-cohomology}.
\end{proof}
Theorem~\ref{thm:vanishing} is a consequence of these two results.
\vanishing*
\begin{proof}
  Cases \ref{thm:vanishing-Mgc} and \ref{thm:vanishing-Mgbar} are immediate from Theorem~\ref{thm:perverse}.\ref{thm:perverse-constructible} and Theorem~\ref{thm:Mgbar}.
  For \ref{thm:vanishing-Mg}, note that $\cM_g[p^n] \subset \cM^{\cc}_g[p^n]$ is the inclusion of the complement of a Cartier divisor and the derived direct image of $\F_p[3g-3]$ is therefore semiperverse by Theorem~\ref{thm:affine-exact}.
  We conclude from Theorem~\ref{thm:perverse}.\ref{thm:perverse-perverse} and smooth base change.
\end{proof}
\begin{exmp}\label{exmp:Mumford-compactification}
  When $p^n \ge 3$, the stack $\cM^{\cc}_g[p^n]$ is isomorphic to its coarse space $M^{\cc}_g[p^n]$.
  Another compactification $\overline{M_g}[p^n]$ of $M^{\cc}_g[p^n]$ due to Mumford is given by the normalization of $\overline{M_g}$ inside the function field of $M^{\cc}_g[p^n]$.
  This is the coarse space of $\overline{\cM_g}[p^n]$.

  Since $\overline{M_g}[p^n]$ is in general singular at the boundary $D_{p^n} \colonequals \overline{M_g}[p^n] \smallsetminus M^{\cc}_g[p^n]$ (as follows for example from (\ref{eqn:local-ring-Mumford-compactification}) below), $D_{p^n}$ with its induced reduced subscheme structure cannot be a normal crossings divisor and it does not make sense to ask if the transition maps of the projective system $\overline{M_g}[p^n]$ are $p$-ramified over $D_{p^n}$.
  However, one might wonder if methods akin to those of \S~\ref{sect:boundary-vanishing} can still be used to prove a version of Theorem~\ref{thm:Mgbar}.
  We show now that this is in fact not the case.

  Let $y \in \overline{M_3}(\C)$ be a closed point corresponding to a ``wheel'' of two smooth genus $1$ curves attached at two points.
  Let $y_n \in \overline{M_3}[p^n](\C)$ be a compatible system of lifts of $y$.
  For any $n$, there is an isomorphism $\theta_n \colon \cO^\wedge_{\overline{\cM_3}[p^n],y_n} \simeq \C \llbracket t_1,\dotsc,t_6 \rrbracket$, with the singular curves parametrized by the locus $\{t_1t_2 = 0\}$.
  Since the vanishing cycles corresponding to the two nodes are homologous, the description of $H_{y_n}$ for $k = \C$ via Dehn twists from Example~\ref{exmp:description-Hy} shows that $H_{y_n}$ is given by the antidiagonal in $\Aut_C(\cC) \simeq \mu_{p^n} \times \mu_{p^n}$.
  By \cite[Lem.~5.3]{MR2309994} (or a direct tangent space calculation), $\theta_n$ can be chosen such that the action of $\bigl(\zeta,\zeta^{-1}\bigr) \in H_{y_n}$ on $\cO^\wedge_{\overline{\cM_3}[p^n],y_n}$ from Lemma~\ref{lem:description-automorphisms} is identified with
  \[ t_i \mapsto \begin{cases} \zeta^{3-2i}t_i & \text{if } 1 \leq i \leq 2, \\
    t_i & \text{if } 3 \leq i \leq 6, \end{cases} \]
  and the transition maps become
  \[ t_i \mapsto \begin{cases} t^p_i & \text{if } 1 \leq i \leq 2, \\
    t_i & \text{if } 3 \leq i \leq 6. \end{cases} \]

  The completed local ring of the coarse moduli space at $y_n$ is computed as the $H_{y_n}$-invariants of $\cO^\wedge_{\overline{\cM_3}[p^n],y_n}$, and thus in these coordinates as the $\C$-subalgebra of $\C \llbracket t_1,\dotsc,t_6 \rrbracket$ generated by $t^{p^n}_1$, $t_1t_2$, $t^{p^n}_2$, and $t_i$ for $i \geq 3$.
  In other words,
  \begin{equation}\label{eqn:local-ring-Mumford-compactification}
    \cO^\wedge_{\overline{M_3}[p^n],y_n} \simeq \C\llbracket t^{p^n}_1,t^{p^n}_2,t_3,\dotsc,t_6,z \rrbracket / \bigl(t^{p^n}_1t^{p^n}_2 - z^{p^n}\bigr),
  \end{equation}
  with the transition maps given as before and by $z \mapsto z^p$.
  Further, the complement of $D_{p^n}$ in $\Spec \cO^\wedge_{\overline{M_3}[p^n],y_n}$ is $\Spec k \llbracket t,t_3,\dotsc,t_6,z \rrbracket_{tz}$ via the isomorphism taking $t^{p^n}_1$ to $t$ and $t^{p^n}_2$ to $\frac{z^{p^n}}{t}$.
  Here, the transition maps are $(t,t_3,\dotsc,t_6,z) \mapsto (t,t_3,\dotsc,t_6,z^p)$.
  Thus,
  \[ \hocolim \bigl(\hat{\imath}_{n,*}\R \hat{\imath}^!_n \widehat{\Lambda}_n\bigr)_{y_n} \simeq \hocolim \R\widetilde{\Gamma}\bigl(S^1 \times S^1,\F^{\oplus r}_p\bigr)[-1] \]
  does not vanish.
\end{exmp}

\section*{Acknowledgments}
I am deeply grateful to my advisor Bhargav Bhatt for introducing me to many of the ideas surrounding this paper and for sharing numerous crucial mathematical insights with me.
I would further like to thank H\'el\`ene Esnault, David Hansen, Shizhang Li, and Martin Olsson for helpful conversations and correspondence.

\end{document}